\documentclass{article}
\usepackage[utf8]{inputenc}
\usepackage{amsmath}
\usepackage{bm}
\usepackage[all,knot]{xy}
\usepackage{amssymb}
\usepackage{textcomp}
\usepackage{mathtools}
\usepackage{ytableau}
\ytableausetup
{mathmode, boxsize=2em}
\usepackage[hang]{footmisc}
\usepackage{multirow}
\usepackage{hyperref}
\usepackage{ dsfont }
\usepackage{ stmaryrd }
\usepackage{tipa}
\usepackage{mathrsfs}
\usepackage{booktabs,siunitx}
\usepackage{ragged2e,tabularx, makecell}
\usepackage{scrextend}
\usepackage[a4paper,
            bindingoffset=0.2in,
            left=1.5in,
            right=1.5in,
            top=1in,
            bottom=1in,
            footskip=.25in]{geometry}
\usepackage{enumerate}
\usepackage{multibib}
\usepackage{blkarray}
\usepackage{bm}

\usepackage{lipsum}

\usepackage{blindtext,titlefoot}

\newcommand{\Eta}{\mathrm{H}}

\usepackage{amsthm}
\newtheorem{theorem}{Theorem}[section]
\newtheorem*{theorem*}{Theorem}
\newtheorem*{lemma*}{Lemma}
\newtheorem*{fact*}{Fact}
\newtheorem*{conjecture*}{Szpiro's conjecture}
\newtheorem{proposition}[theorem]{Proposition}
\newtheorem{lemma}[theorem]{Lemma}

\newtheorem{corollary}[theorem]{Corollary}
\newtheorem{conjecture}[theorem]{Conjecture}

\newtheorem{lemdef}[theorem]{Lemma-Definition}

\newtheorem*{maintheorem}{Main theorem}

\theoremstyle{remark}
\newtheorem{remark}[theorem]{Remark}

\numberwithin{equation}{section}

\begin{document}
\title{Ramification of Tate modules for rank 2 Drinfeld modules}
\author{Takuya Asayama and Maozhou Huang\footnote{\texttt{huang.maozhou@outlook.com}}}

\maketitle

\unmarkedfntext{\text{2020\it{  Mathematics Subject Classification}}. 11G09, 11S15, 11S20, 11F80.}

\unmarkedfntext{\text{\it{Key words and phrases}}: Drinfeld modules, Herbrand functions, higher ramification subgroups, conductors, Szpiro's conjecture.}

\begin{abstract}
In this paper, we study the ramification of extensions of a function field generated by division points of rank 2 Drinfeld modules.
Also conductors of certain rank 2 Drinfeld modules are defined as analogues of those for elliptic curves. 
A calculation of these conductors allows us to show an analogue of Szpiro's conjecture under a certain limited situation.
\end{abstract}

\section{Introduction} 

\subsection{The Galois representation of the wild ramification subgroup on the division points}\label{s11}

Let us first introduce the notation used throughout this paper.
Put $\Bbb{F}_{q}[t]$ to be the polynomial ring in $t$ over the field $\Bbb{F}_{q}$ of order $q=p^{s}$ which is a power of a rational prime $p.$
Let $F$ be a finite extension of the fraction field $\Bbb{F}_{q}(t)$ of $\Bbb{F}_{q}[t].$ 
Let $\phi$ be a rank $2$ Drinfeld $\Bbb{F}_{q}[t]$-module over some extension $L$ over $F.$ 
It is characterized by the polynomial $\phi_{t}(X)=tX+a_{1}X^{q}+a_{2}X^{q^{2}}\in L[X].$
Let $\bm{j}=\bm{j}(\phi)$ denote the $j$-invariant $a_{1}^{q+1}/a_{2}$ of $\phi.$ 

Let $v$ be a prime of $F$ and let it also denote the valuation associated to $v$ normalized so that $v(F^{\times})=\Bbb{Z}.$ 
Let $K$ be the completion of $F$ at $v.$ 
In this paper, $v$ is uniquely extended to the fixed separable closure $K^{\text{sep}}/K$ even if nothing is mentioned.
For a finite separable extension $L/K,$ let $v_{L}\coloneqq e_{L/K}\cdot v$ be the normalized valuation of $L,$ where $e_{L/K}$ is the ramification index of $L/K.$ 
We have $v_{L}(L^{\times})=\Bbb{Z}.$
An infinite prime of $F$ is a prime of $F$ lying above the prime $(1/t)$ of $\Bbb{F}_{q}(t).$
A finite prime of $F$ is a prime which is not infinite. 
For an element $a\in \Bbb{F}_{q}[t],$ let $\phi[a]$ denote the set of the roots of $\phi_{a}(X).$
We denote the Galois group of a Galois extension $L/K$ by $G(L/K).$ We denote the cardinality of a set $S$ by $\#S.$

Let $\pi\in \Bbb{F}_{q}[t]$ be a degree $1$ polynomial and $n$ a positive integer.
In this paper, we first investigate the structure of the wild ramification subgroup of the extension $K(\phi[\pi^{n}])/K$ and how the elements of this group act on $K(\phi[\pi^{n}]).$ Our main result is 
\begin{maintheorem}[Theorems~\ref{p2211}~(2) and \ref{p2221}~(2)]
Let $v$ be a prime of $F$ and $n$ a positive integer.
Let $\phi$ be a Drinfeld $\Bbb{F}_{q}[t]$-module of rank $2$ over $K.$
Let $\pi$ be a degree $1$ polynomial in $\Bbb{F}_{q}[t]$ such that $v\nmid \pi.$ 
Assume
\begin{equation}
    \begin{cases}
    \begin{split}
    &p\nmid v(\bm{j})\text{ and }v(\bm{j})<v(\pi)q\text{ if }v\text{ is infinite}\text{\rm{;}}
    \end{split}\\
    \begin{split}&q\neq 2,\,\,\,p\nmid v(\bm{j}),\text{ and }v(\bm{j})<0\text{ if }v\text{ is finite.}\end{split}
    \end{cases}\nonumber
\end{equation}
Under this assumption, there exists \text{\rm{(}}see Propositions~\text{\rm{\ref{p2111}}} and \text{\rm{\ref{p2121}~(1)}}\text{\rm{)}} an $\Bbb{F}_{q}$-basis $\\\{\xi_{-n},\ldots, \xi_{-1},\xi_{n},\ldots, \xi_{1}\}$ of $\phi[\pi^{n}]$ such that\text{\rm{:}}
\begin{itemize}
\item $\xi_{-1}$ and $\xi_{1}$ are nonzero roots of $\phi_{\pi}(X)$ 
such that $v(\xi_{-1})>v(\xi_{1});$ 

\item each $\xi_{-i}$ \text{\rm{(}}resp.\ $\xi_{i}$\text{\rm{)}} is a root of $\phi_{\pi}(X)-\xi_{-(i-1)}$ \text{\rm{(}}resp.\ $\phi_{\pi}(X)-\xi_{i-1}$\text{\rm{)}} with the largest valuation among the valuations of all roots of $\phi_{\pi}(X)-\xi_{-(i-1)}$ \text{\rm{(}}resp.\ $\phi_{\pi}(X)-\xi_{i-1}$\text{\rm{)}}.
\end{itemize}
\begin{enumerate}[\rm{(}1)]
    \item 
Assume that $v$ is infinite. 
Let $m$ be the integer such that $v(\bm{j})\in (v(\pi)q^{m+1},v(\pi)q^{m}).$ 
For $n\geq m,$ the wild ramification subgroup of $K(\phi[\pi^{n}])/K$ is identified with that of $K(\phi[\pi^{m}])/K$ and is isomorphic to $(\Bbb{Z}/p\Bbb{Z})^{sm}.$ 
It is generated by $\sigma_{l,u}$ for $l=1,\ldots,m$ and $u\in\Bbb{F}_{q}$ such that each $\sigma_{l,u}$ acts on the basis $\{\xi_{-m},\ldots,\xi_{-1},\xi_{m},\ldots,\xi_{1}\}$ of $\phi[\pi^{m}]$ via the matrix with coefficients in $\Bbb{F}_{q}$
\[\begin{pmatrix}I_{m}&0\\
u\cdot A_{m,l}&I_{m}
\end{pmatrix},
\]
where $I_{m}$ denotes the $m\times m$ identity matrix and $A_{m,l}$ is the $m\times m$ matrix defined by $(\delta_{i,j-l+1})_{ij}$ with the Kronecker delta $\delta.$
\item
Assume that $v$ is finite. 
For any $n,$ the wild ramification subgroup of $K(\phi[\pi^{n}])/K$ is isomorphic to $(\Bbb{Z}/p\Bbb{Z})^{sn}.$ 
It is generated by $\sigma_{l,u}$ for $l=1,\ldots,n$ and $u\in \Bbb{F}_{q}$ such that each $\sigma_{l,u}$ acts on the basis $\{\xi_{-n},\ldots,\xi_{-1},\xi_{n},\ldots,\xi_{1}\}$ of $\phi[\pi^{n}]$ via the matrix with coefficients in $\Bbb{F}_{q}$
\[\begin{pmatrix}I_{n}&0\\
u\cdot A_{n,l}&I_{n}
\end{pmatrix}
\]
for $A_{n,l}$ and $I_{n}$ defined as in \text{\rm{(1)}}.
\end{enumerate}
\end{maintheorem}

We remark that this theorem heavily relies on the valuations of elements in $\phi[\pi^{n}]$ for any positive integer $n.$
These valuations are studied in Section~\ref{s21} if the degree of $\pi$ is $1.$
The second author has generalized these results when the degree of $\pi\geq 2.$
However, there are still difficulties in generalizing the results on the ramification of $\phi[\pi^{n}].$ 
We do not pursue the general case in the present paper, but we hope to study this question further in the future.

\subsection{Review on conductors of elliptic curves}\label{s12}
Before we turn to the application of Main theorem, let us recall some facts about the theory of elliptic curves in this subsection.

Let $E$ be an elliptic curve over a local number field $K$ of residue characteristic $p>0$. 
For a prime number $\ell\nmid p,$ let $E[\ell]$ denote the $\Bbb{F}_{\ell}$-vector space of the $\ell$-division points of $E.$
Let $G_{i}$ (resp.\ $G^{y}$) denote the $i$-th lower (resp.\ $y$-th upper) ramification subgroup of the Galois group of the extension $K(E[\ell])/K$ generated by the $\ell$-division points of $E.$ 
Define the wild part of the conductor of $E/K$ to be the quantity 
\begin{align}\delta(E/K)\coloneqq
\int_{0}^{+\infty}\frac{\#\,G_{i}}{\#\,G_{0}}\text{codim}_{\Bbb{F}_{\ell}}(E[\ell]^{G_{i}})di=\int_{0}^{+\infty}\text{codim}_{\Bbb{F}_{\ell}}(E[\ell]^{G^{y}})dy,\nonumber\end{align}
where $E[\ell]^{G_{i}}$ is the subspace of elements of $E[\ell]$ fixed by $G_{i}$ and $E[\ell]^{G^{y}}$ is similarly defined.
Define the tame part of the conductor to be
\[\varepsilon(E/K)\coloneqq \text{codim}_{\Bbb{Q}_{\ell}}(V_{\ell}(E)^{I(K^{\text{sep}}/K)})=\begin{cases}
0&E\text{ has good reduction};\\
1&E\text{ has multiplicative reduction};\\
2&E\text{ has additive reduction.}
\end{cases}\]
Here $V_{\ell}(E)=\varprojlim E[\ell^{n}]\otimes_{\Bbb{Z}_{\ell}}\Bbb{Q}_{\ell}$ denotes the rational $\ell$-adic Tate module and $I(K^{\text{sep}}/K)$ is the inertia subgroup of the absolute Galois group of $K.$
Put as in \cite[p.\,380]{Sil} 
\begin{align}f(E/K)=\delta(E/K)+\varepsilon(E/K)\label{f121}\end{align}
to be the conductor of $E$ over $K.$
The quantity $f(E/K)$ is an integer which is independent of the choice of $\ell.$

Let $E$ be an elliptic curve over a (global) number field $F.$ The conductor of $E/F$ is the ideal $\mathfrak{f}(E/F)=\prod_{\mathfrak{p}\text{ finite}}\mathfrak{p}^{f(E/F_{\mathfrak{p}})}$ defined by all conductors of $E/F_{\mathfrak{p}},$ where $F_{\mathfrak{p}}$ is the completion of $F$ at $\mathfrak{p}.$ 
Here the product extends over all finite primes $\mathfrak{p}$ of $F.$  
The conductor measures the extent to which an elliptic curve has bad reduction.

There is another invariant, called the minimal discriminant, which measures how bad the reduction is.
The minimal discriminant $\mathcal{D}(E/F)$ of $E/F$ is the product of the minimal discriminants of integral models of $E/F_{\mathfrak{p}}$ for all finite primes $\mathfrak{p}$ of $F.$
Szpiro proposed a conjecture (see \cite[p.\,10]{Sz} or \cite[Chapter~IV, Szpiro's Conjecture~10.6]{Sil}) concerning a relation between these two invariants. 
A stronger form of this conjecture was proposed by Lockhart-Rosen-Silverman in \cite[Remark 5]{LRS}.
\begin{conjecture*}
\begin{enumerate}[\rm{(}1)]
\item 
Fix a number field $F$ and a real positive number $\varepsilon.$ 
Then there exists a constant $C(F,\varepsilon)$ such that\text{\rm{:}} for any elliptic curve $E$ over $F,$ its minimal discriminant $\mathcal{D}(E/F)$ and its conductor $\mathfrak{f}(E/F)$ satisfy
\[N_{F/\Bbb{Q}}(\mathcal{D}(E/F))\leq C(F,\varepsilon)(N_{F/\Bbb{Q}}(\mathfrak{f}(E/F)))^{6+\varepsilon}.\]

\item \text{\rm{(Stronger form)}} 
Put $\mathfrak{f}^{\text{\rm{tame}}}(E/F)\coloneqq \prod_{\mathfrak{p}\text{\rm{ finite}}}\mathfrak{p}^{\varepsilon(E/F_{\mathfrak{p}})}.$ Then there exists a constant $C(F,\varepsilon)$ such that for any elliptic curve $E$ over $F,$ 
\[N_{F/\Bbb{Q}}(\mathcal{D}(E/F))\leq C(F,\varepsilon)(N_{F/\Bbb{Q}}(\mathfrak{f}^{\text{\rm{tame}}}(E/F)))^{6+\varepsilon}.\]
\end{enumerate}
\end{conjecture*}

The conductor $f(E/K)$ of $E$ over a local number field $K$ is estimated by 
\begin{theorem}[\text{Lockhart-Rosen-Silverman \cite{LRS}, Brumer-Kramer \cite{BK}}]\label{t121}
Let $K/\Bbb{Q}_{p}$ be a local field with normalized valuation $v_{K},$ and let $E/K$ be an elliptic curve. 
Then $f(E/K)$ has an upper bound 
\[f(E/K)\leq 2+3v_{K}(3)+6v_{K}(2).\]
\end{theorem}
\noindent See also \cite[Chapter~IV,\,Theorem~10.4]{Sil}. 
This estimate can be very important for the study of Szpiro's conjecture since it implies that (1) and (2) in the above conjecture are equivalent.

For an abelian variety $A$ over a local number field $K,$ its conductor $f(A/K)$ is defined in \cite[(12), (13)]{LRS}. The definition is similar to (\ref{f121}). 
Lockhart-Rosen-Silverman also proposed a ``partial generalization of Szpiro's conjecture''. 
\begin{conjecture}[\text{\cite[(10)]{LRS}}]\label{c11}
Let $A$ be an abelian variety of dimension $d$ over a number field $F.$ Put $\mathfrak{f}(A/F)\coloneqq \prod_{\mathfrak{p}\text{\rm{ finite}}}\mathfrak{p}^{f(A/F_{\mathfrak{p}})}.$ 
Let $h_{\text{\rm{FP}}}(A/F)$ denote the Faltings-Parshin height of $A/F.$ Then there are constants $C_{1}(F,d)$ and $C_{2}(F,d),$ depending only on $F$ and $d,$ such that 
\[h_{\text{\rm{FP}}}(A/F)\leq C_{1}(F,d)\cdot \text{\rm{log}}|N_{F/\Bbb{Q}}\mathfrak{f}(A/F)|+C_{2}(F,d).\]
\end{conjecture}
Szpiro's conjecture for $\mathbb{Q}$ follows from the $abc$ conjecture \cite[Section~2, Remarque]{Sz}.
Mochizuki \cite{Mo} announced the proof of the $abc$ conjecture via the inter-universe Teichm\"{u}ller theory.

\subsection{An analogue of Szpiro's conjecture}\label{s13}

Fix the notation used in Section~\ref{s11}.
Let $T_{\pi}$ be the $\pi$-adic Tate module of $\phi.$ 
We require $v\nmid \pi.$ Recall that $\pi$ is of degree $1.$ 

Rank $2$ Drinfeld modules over $F$ are considered as analogues of elliptic curves over a number field.
Our next goal is to find a function field analogue of Conjecture~\ref{c11}. 
To do this, we first need to define an analogue of the conductor above. 
The estimate in Theorem~\ref{t121} suggests that when working with Szpiro's conjecture and its variant for elliptic curves over a number field, one may ignore the contribution of wild ramification. 
On the contrary, for the extensions generated by division points of Drinfeld modules, the wild ramification can be made arbitrarily large (see Remark \ref{r312}).
So it is worth investigating a relation between the height and the wild part of the conductor of a Drinfeld module.

Rather than the vector space of division points, we consider the $G_{v}$-module $T_{\pi},$ where $G_{v}$ denotes the absolute Galois group of $K.$
Since $G_{v}$ is a profinite group, a definition similar to that of ``$\delta(E/K)$'' in (\ref{f121}) using lower ramification subgroups is not valid.

Using the notion of the upper ramification subgroups, we define for a rank $2$ Drinfeld $\Bbb{F}_{q}[t]$-module $\phi$ over $K$ the quantity
\[\mathfrak{f}_{v}(\phi)\coloneqq \int_{0}^{+\infty}(2-\text{rank}_{\Bbb{F}_{q}[t]_{\pi}}T_{\pi}^{G_{v}^{y}})dy\]
as an analogue of the wild part of the conductor $\delta(E/K)$ of an elliptic curve.
Here $G_{v}^{y}$ denotes the $y$-th upper ramification subgroup of $G_{v}.$
Note that the prime $v$ can be infinite or finite. 
In fact, the infinite part of a height (e.g.,\ $J$-height) is not bounded. 
When we want to relate the height to the conductor (as in Theorem~\ref{t131}) in the function field case, we must define the conductors at infinite primes, unlike in the number field case.

\begin{lemma}[Lemma-Definition~\ref{d311} and Corollary~\ref{c311}]\label{l131} 
Let $\phi$ be a rank $2$ Drinfeld $\Bbb{F}_{q}[t]$-module over $F.$
For a prime $v$ of $F,$ let $\pi$ be a degree $1$ polynomial in $\Bbb{F}_{q}[t]$ with $v\nmid \pi.$ 
Assume 
\begin{equation}
    \begin{cases}
    \begin{split}
    &\text{either }\big(p\nmid v(\bm{j})\text{ and }v(\bm{j})<v(\pi)q\big),&\\
    &\quad\quad\quad\quad\text{or }v(\bm{j})\geq v(\pi)q\text{ if }v\text{ is infinite};
    \end{split}\\
    \begin{split}&\text{either }\big(q \neq 2,\,\,p\nmid v(\bm{j}),\text{ and }v(\bm{j})<0\big),&\\
    &\quad\quad\quad\quad\text{or }v(\bm{j})\geq 0\text{ if }v\text{ is finite}.\end{split}
    \end{cases}\nonumber
\end{equation}
Then the integral $\mathfrak{f}_{v}(\phi)$ is independent of the choice of the degree $1$ polynomial $\pi$ and we have
\begin{align}\mathfrak{f}_{v}(\phi)=\begin{cases}
\begin{cases}
0 & v(\bm{j})\in [v(\pi)q,+\infty);\\
\frac{-v(\bm{j})+v(\pi)q}{q-1}& v(\bm{j})\in(-\infty,v(\pi)q),\,\,p\nmid v(\bm{j}),
\end{cases}&v\text{ is infinite};\\
\begin{cases}
0&v(\bm{j})\in [0,+\infty);\\
\frac{-v(\bm{j})}{q-1}&v(\bm{j})\in (-\infty,0),\,\,p\nmid v(\bm{j}),\text{ and }q\neq 2,
\end{cases}&v\text{ is finite}.
\end{cases}\nonumber\end{align}
\end{lemma}
\noindent 
The case where $v(\bm{j})\geq v(\pi)q$ and $v$ is infinite, and the case where $v(\bm{j})\geq 0$ and $v$ is finite follow from the fact that the extension $K(\phi[\pi^{n}])/K$ is tamely ramified (see Lemmas~\ref{l231} and \ref{l232}).
The other cases follow from Main theorem. 

For a prime $v$ of $F$ and the completion $F_{v},$ let $\text{deg}(v)$ denote the degree of the residue field of $F_{v}$ over $\Bbb{F}_{q}.$
As $\text{deg}\,\pi=1$ is assumed, we call $\mathfrak{f}(\phi)\coloneqq \sum_{v}\text{deg}(v)\cdot \mathfrak{f}_{v}(\phi)$ the \text{\it{global }}$1$\text{\it{-conductor}} of $\phi,$ where the sum extends over all primes $v$ of $F.$

Since the conductors for certain Drinfeld modules are almost given by the $j$-invariant, the $J$-height (see (\ref{f412}) for definition) of a rank $2$ Drinfeld module defined in Breuer-Pazuki-Razafinjatovo \cite{BPR} using the $j$-invariant can be a good replacement of the Faltings-Parshin height.
By the above lemma, we then obtain the following analogue of Conjecture~\ref{c11}:

\begin{theorem}[Theorem~\ref{p311}]\label{t131}
Make the same assumptions as in the above lemma. Let $h_{J}(\phi)$ denote the $J$-height of $\phi.$ Then we have the inequality
\[h_{J}(\phi)\leq \mathfrak{f}(\phi)\cdot \frac{q-1}{[F:\Bbb{F}_{q}(t)]}+q.\]
\end{theorem}

We may define the conductor (or $1$-conductor) for Drinfeld modules with a higher rank.
We may expect that there is a similar statement to this theorem for higher rank Drinfeld modules. 
However, we do not know much about the generalization.

\subsection{Contents}
Keep the notation (for function fields) above.
Assume that $\pi$ is of degree $1$ and $v\nmid\pi.$ 
We may assume that $\pi$ is monic and we have 
\[\phi_{t}(X)=tX+a_{1}X^{q}+a_{2}X^{q^{2}},\,\,\phi_{\pi}(X)=\pi X+a_{1}X^{q}+a_{2}X^{q^{2}}.\]
For $v$ an infinite prime or a finite prime not dividing $t$ and $\pi,$ the coefficients of $X^{q^{i}}$ in $\phi_{t}(X)$ and $\phi_{\pi}(X)$ have the same valuations at $v.$
It turns out that the results in Sections~\ref{s21} and \ref{s22} rely only on the valuations of coefficients of $\phi_{t}(X).$
This allows us to reduce to considering the case where $\pi=t.$ 
In Section~\ref{s21}, we investigate the structure of $\phi[t^{n}]$ and the valuations of its elements.
In Section~\ref{s22}, we show in Lemma~\ref{l2213} that the extension $K(\phi[t^{n+1}])/K(\phi[t^{n}])$ is generated by a root of some degree $q$ polynomial. Lemma~\ref{l2214} is a similar result concerning the extension $K(\phi[t])/K.$ 
Thanks to Proposition~\ref{p221} on the ramification of the extension generated by the roots of $X^{q}-X-\alpha,$ 
we are able to obtain the ramification breaks of $K(\phi[t^{n+1}])/K(\phi[t^{n}])$ and of $K(\phi[t])/K.$
Then we prove Main theorem. 
In Section~\ref{s311}, we define the conductors under the assumption of Lemma~\ref{l131}. We show an analogue of Conjecture~\ref{c11} in Section~\ref{s312}. 
Appendix \ref{s213} is a supplement to Section~\ref{s21}, where we resume the study of the structure of $\phi[t^{n}]$ and the valuations of its elements when $v$ divides $t.$

We remark that Lemmas~\ref{l2111}, \ref{l2211}, and \ref{l231}~(2) which concern the valuations of the $t^{n}$-division points and the extension $K(\phi[t^{n}])/K,$ where $K$ is the completion of $F$ at an infinite prime, have appeared in the paper of Maurischat~\cite{Mau}.

\section*{Acknowledgments}
The authors express their sincere gratitude to Professor Yuichiro Taguchi for introducing them the initial idea.
His advice and encouragement have been invaluable to them. 
We greatly appreciate the anonymous referee for a careful reading and many thoughtful comments that helped us improve our manuscript.

\section{Valuations of the $t^{n}$-division points}\label{s21}
In this section, let $v$ denote a finite or infinite prime of $F$ which \text{\it{does not}} divide $t.$
We also denote the corresponding valuation by $v$ so that the uniformizer of this prime has valuation $1$ under $v.$
Let $\phi$ be a rank $2$ Drinfeld $\Bbb{F}_{q}[t]$-module over $K.$
Denote by $K_{n}$ the extension of $K$ generated by all roots of $\phi_{t^{n}}(X)$ for an integer $n\geq 0,$ i.e., $K_{n}=K(\phi[t^{n}]).$  

The $\Bbb{F}_{q}$-vector spaces $\phi[t^{i}]$ for $i=0,1,2,\ldots$ form a sequence $\{0\}= \phi[t^{0}]\subset \phi[t^{1}]\subset \phi[t^{2}]\subset\cdots.$ For each $i,$ the vector space $\phi[t^{i}]$ is formed by the roots of all polynomials $\phi_{t}(X)-\xi$ for $\xi\in \phi[t^{i-1}].$
In this section, we obtain the valuations of elements of $\phi[t^{i}]$ from those of elements in $\phi[t^{i-1}]$ and finally obtain those of elements in $\phi[t^{n}].$  
\subsection{Infinite primes}\label{s211}

In this subsection, we denote $v$ by $v_{\infty}$ to emphasize we are working with an infinite prime.
Let $v_{0},$ $v_{1},$ and $v_{2}$ denote respectively the valuations of $t,$ $a_{1},$ and $a_{2}$ in $\phi_{t}(X)=tX+a_{1}X^{q}+a_{2}X^{q^{2}}.$
In the present case, we have $v_{\infty}(t)=v_{0}<0.$

Let from now on $P_{0}=(1,v_{0}),$ $P_{1}=(q, v_{1}),$ and $P_{2}=(q^2, v_{2}).$ 
We define $\mu(P, Q)$ to be the slope of the segment $PQ$ for $P,$ $Q\in \Bbb{R}^{2}.$ 
We have $\mu(P_{0},P_{1})=\frac{v_{1}-v_{0}}{q-1},$ $\mu(P_{0},P_{2})=\frac{v_{2}-v_{0}}{q^2-1},$ and 
\begin{align}\mu(P_{0},P_{1})-\mu(P_{0}, P_{2})=\frac{v_{\infty}(\bm{j})-v_{0}q}{q^2-1}.\label{f2111}\end{align} 
We see that $\mu(P_{0},P_{1})<\mu(P_{0}, P_{2})$ if and only if $v_{\infty}(\bm{j})<v_{0}q.$ 
Assume $v_{\infty}(\bm{j})<v_{0}q.$ 
Then the Newton polygon of $\phi_{t}(X)$ is $P_{0}P_{1}P_{2}$ having exactly two segments $P_{0}P_{1}$ and $P_{1}P_{2}.$
We have $\mu(P_{1}, P_{2})=\frac{v_{2}-v_{1}}{(q-1)q}.$ 
Let $\xi_{-1}$ and $\xi_{1}$ be two roots of $\phi_{t}(X)$ with valuations
\[v_{\infty}(\xi_{-1})=-\mu(P_{0},P_{1})=-\frac{v_{1}-v_{0}}{q-1},\,\,v_{\infty}(\xi_{1})=-\mu(P_{1},P_{2})=-\frac{v_{2}-v_{1}}{(q-1)q}.\]
The roots $\xi_{-1}$ and $\xi_{1}$ generate the $\Bbb{F}_{q}$-vector space $\phi[t].$

The roots of $\phi_{t^{2}}(X)$ are those of the equations \begin{align}\phi_{t}(X)=u_{-1}\cdot \xi_{-1}+u_{1}\cdot \xi_{1}\text{ for all }u_{-1},\,\,u_{1}\in \Bbb{F}_{q}.\nonumber\end{align}
Let $x_{-}$ (resp.\ $x_{+}$) be any root of $\phi_{t}(X)=\xi_{-1}$ (resp.\ $\phi_{t}(X)=\xi_{1}$).
Then $u_{-1}\cdot x_{-}+u_{1}\cdot x_{+}$ is a root of the polynomial $\phi_{t}(X)=u_{-1}\cdot \xi_{-1}+u_{1}\cdot \xi_{1}.$
To understand the valuations of the elements in $\phi[t^{2}],$ it is enough to work with $\phi_{t}(X)=\xi_{-1}$ and $\phi_{t}(X)=\xi_{1}.$

In general, for $n\geq 1,$ denote by $\xi_{-(n+1)}$ (resp.\ $\xi_{n+1}$) a root of $\phi_{t}(X)=\xi_{-n}$ (resp.\ $\phi_{t}(X)=\xi_{n}$) such that $v_{\infty}(\xi_{-(n+1)})$ (resp.\ $v_{\infty}(\xi_{n+1})$) is the largest among the valuations of all roots of $\phi_{t}(X)=\xi_{-n}$ (resp.\ $\phi_{t}(X)=\xi_{n}$).
In this way, the elements $\xi_{-n}$ and $\xi_{n}$ for $n\geq 1$ in $K^{\text{sep}}$ are inductively defined. The next lemma concerns their valuations.
\begin{lemma}\label{l2111}
Assume $v_{\infty}(\bm{j})<v_{0}q.$ 
Let $m$ be the integer such that $v_{\infty}(\bm{j})\in (v_{0}q^{m+1},v_{0}q^{m}].$

\begin{enumerate}[\rm{(}1)]
\item
We have 
\[v_{\infty}(\xi_{n})=\begin{cases}
-\dfrac{v_{2}+v_{1}(q^{n}-q-1)}{(q-1)q^{n}} & 1\leq n\leq m;\\
- \left(v_{0}(n-m) + \dfrac{v_{2} + v_1 (q^m - q -1) }{ (q-1)q^m}\right) & n\geq m+1.
\end{cases}\]
Put $Q_{n}=(0,v_{\infty}(\xi_{n})).$
If $1 \le n \le m-1,$ then the Newton polygon of $\phi_{t}(X)-\xi_{n}$ is
$Q_{n}P_{1}P_{2}$ having exactly two segments. If $n \ge m,$ then the Newton polygon of $\phi_{t}(X)-\xi_{n}$ is $Q_{n}P_{0}P_{1}P_{2}$ having exactly three
segments. 

\item
For $n\geq 1,$ we have  \[v_{\infty}(\xi_{-n})=
-\left(v_{0}(n-1)+\dfrac{v_{1}-v_{0}}{q-1}\right).\]
Put $Q_{-n}=(0,v_{\infty}(\xi_{-n})).$
Then the Newton polygon of $\phi_{t}(X)-\xi_{-n}$ is $Q_{-n}P_{0}P_{1}P_{2}$ having exactly three segments.
\end{enumerate}
\end{lemma}

\begin{proof}

We check (1). 
Let us first consider the Newton polygon of $\phi_{t}(X)-\xi_{1}.$
We have 
\[\mu(Q_{1},P_{0})=v_{0}+\frac{v_{2}-v_{1}}{(q-1)q},\,\,\mu(Q_{1},P_{1})=\frac{v_{1}+\frac{v_{2}-v_{1}}{(q-1)q}}{q-0}
,\,\,\mu(Q_{1},P_{2})=\frac{v_{2}+\frac{v_{2}-v_{1}}{(q-1)q}}{q^{2}-0}.\]
We calculate
\begin{align}
\mu(Q_{1},P_{0})-\mu(Q_{1},P_{1})&=\frac{-v_{\infty}(\bm{j})+v_{0}q^{2}}{q^{2}},\nonumber\\
\mu(Q_{1},P_{1})-\mu(Q_{1},P_{2})&=\frac{v_{\infty}(\bm{j})(q-1)}{q^{3}}<0.\nonumber\end{align}
If $v_{\infty}(\bm{j})\in (v_{0}q^{2},v_{0}q),$ we have $\mu(Q_{1},P_{0})<\mu(Q_{1},P_{1})$ and then $P_{0}$ is a vertex of the Newton polygon. 
In this case, the Newton polygon is $Q_{1}P_{0}P_{1}P_{2}$ having exactly three segments by the argument below (\ref{f2111}).
If $v_{\infty}(\bm{j})\leq v_{0}q^{2},$ we have $\mu(Q_{1},P_{0})\geq \mu(Q_{1},P_{1})$ and the Newton polygon is $Q_{1}P_{1}P_{2}$ having exactly two segments.

Assume that (1) for $n-1$ is valid.
If $n\leq m-1,$ the valuation of $\xi_{n}$ is \[-\mu(Q_{n-1},P_{1})=-\frac{v_{1}-v_{\infty}(\xi_{n-1})}{q-0}=-\frac{v_{2}+v_{1}(q^{n}-q-1)}{(q-1)q^{n}}.\]
Next, we determine the Newton polygon of $\phi_{t}(X)-\xi_{n}.$ 
We have
\begin{align}\mu(Q_{n},P_{0})=v_{0}-v_{\infty}(\xi_{n}),\,\,\mu(Q_{n},P_{1})=\frac{v_{1}-v_{\infty}(\xi_{n})}{q-0}
,\,\,\mu(Q_{n},P_{2})=\frac{v_{2}-v_{\infty}(\xi_{n})}{q^{2}-0}.\label{fl215}\end{align}
We calculate
\begin{align}\mu(Q_{n},P_{0})-\mu(Q_{n},P_{1})&=\frac{-v_{\infty}(\bm{j})+v_{0}q^{n+1}}{q^{n+1}},\label{fl213}\\
\mu(Q_{n},P_{1})-\mu(Q_{n},P_{2})&=\frac{v_{\infty}(\bm{j})(q^{n}-1)}{q^{n+2}}<0.\label{fl214}
\end{align}
Since $n\leq m-1,$ we have $\mu(Q_{n},P_{0})\geq \mu(Q_{n},P_{1}).$
This implies that $Q_{n}P_{1}$ is the first segment of the Newton polygon and the Newton polygon is $Q_{n}P_{1}P_{2}$ having exactly two segments. 

When $n=m,$ we have the same inductive hypothesis as above and $v_{\infty}(\xi_{m})=-\mu(Q_{m-1},P_{1}).$ 
However, we have $\mu(Q_{m},P_{0})<\mu(Q_{m},P_{1})$ by (\ref{fl213}).
Thus $P_{0}$ is a vertex of the Newton polygon of $\phi_{t}(X)-\xi_{m}$ by (\ref{fl214}). 
By the argument below (\ref{f2111}), the Newton polygon is $Q_{m}P_{0}P_{1}P_{2}$ having exactly three segments.

If $n\geq m+1,$ then the valuation of $\xi_{n}$ is calculated by $-\mu(Q_{n-1},P_{0})=-v_{0}+v_{\infty}(\xi_{n-1}).$ 
We show that the Newton polygon of $\phi_{t}(X)-\xi_{n}$ is $Q_{n}P_{0}P_{1}P_{2}$ having exactly three segments.
We have $\mu(Q_{n},P_{0}),$ $\mu(Q_{n},P_{1}),$ and $\mu(Q_{n},P_{2})$ as in (\ref{fl215}).
We calculate
\begin{align}
\mu(Q_{n},P_{0})-\mu(Q_{n},P_{1})&=\frac{-v_{\infty}(\bm{j})+v_{0}(q^{m+1}(n-m+1)-q^{m}(n-m))}{q^{m+1}}<0,\nonumber\\
\mu(Q_{n},P_{1})-\mu(Q_{n},P_{2})&=\frac{v_{\infty}(\bm{j})(q^{m}-1)+v_{0}(n-m)(q^{m+1}-q^{m})}{q^{m+2}}<0.\nonumber
\end{align}
Then $P_{0}$ is a vertex of the Newton polygon. 
By the argument below (\ref{f2111}), the Newton polygon has the desired form.

For (2), we first show that the Newton polygon of $\phi_{t}(X)-\xi_{-1}$ is $Q_{-1}P_{0}P_{1}P_{2}$ having exactly three segment.
We have 
\[\mu(Q_{-1},P_{0})=v_{0}+\frac{v_{1}-v_{0}}{q-1},\,\,\mu(Q_{-1},P_{1})=\frac{v_{1}+\frac{v_{1}-v_{0}}{q-1}}{q-0},\,\,\mu(Q_{-1},P_{2})=\frac{v_{2}+\frac{v_{1}-v_{0}}{q-1}}{q^{2}-0}.\]
We calculate
\begin{align}
\mu(Q_{-1},P_{0})-\mu(Q_{-1},P_{1})&=\frac{v_{0}(q-1)}{q}<0,\nonumber\\
\mu(Q_{-1},P_{1})-\mu(Q_{-1},P_{2})&=\frac{v_{\infty}(\bm{j})-v_{0}}{q^{2}}<0.\nonumber
\end{align}
Then $P_{0}$ is a vertex of the Newton polygon. 
By the argument below (\ref{f2111}), the Newton polygon has the desired form.
Assume that (2) for $n-1$ is valid.
Then the valuation of $\xi_{-n}$ is calculated by $-\mu(Q_{-(n-1)},P_{1})=-v_{0}+v_{\infty}(\xi_{-(n-1)}).$
We show that the Newton polygon of $\phi_{t}(X)-\xi_{-n}$ is $Q_{-n}P_{0}P_{1}P_{2}$ having exactly three segments.
We have
\[\mu(Q_{-n},P_{0})=v_{0}-v_{\infty}(\xi_{-n}),\,\,\mu(Q_{-n},P_{1})=\frac{v_{1}-v_{\infty}(\xi_{-n})}{q-0},\,\,\mu(Q_{-n},P_{2})=\frac{v_{2}-v_{\infty}(\xi_{-n})}{q^{2}-0}.\]
We calculate
\begin{align}
\mu(Q_{-n},P_{0})-\mu(Q_{-n},P_{1})&=\frac{v_{0}n(q-1)}{q}<0,\nonumber\\
\mu(Q_{-n},P_{1})-\mu(Q_{-n},P_{2})&=\frac{v_{\infty}(\bm{j})+v_{0}((n-1)(q-1)-1)}{q^{2}}<0.\nonumber
\end{align}
Then $P_{0}$ is a vertex of the Newton polygon. By the argument below (\ref{f2111}), the Newton polygon has the desired form.
\end{proof}

Put $\mathcal{I}_{n}\coloneqq \{\pm 1, \ldots, \pm n\}.$
We obtained $2n$ elements $\xi_{i}$ for $i\in \mathcal{I}_{n}$ in the $2n$-dimensional $\Bbb{F}_{q}$-vector space $\phi[t^{n}].$ 
The next proposition claims that they actually form a basis of $\phi[t^{n}]$ and can be arranged with respect to their valuations.

\begin{proposition}\label{p2111}
Assume $v_{\infty}(\bm{j})<v_{0}q.$ Let $m$ be the integer such that $v_{\infty}(\bm{j})\in (v_{0}q^{m+1},v_{0}q^{m}].$
Then $\xi_{i}$ for $i\in \mathcal{I}_{n}$ form a basis of $\phi[t^{n}].$
If $n\leq m,$ then we have
\begin{align}v_{\infty}(\xi_{-n})>v_{\infty}(\xi_{-(n-1)})>\cdots>v_{\infty}(\xi_{-1})\geq v_{\infty}(\xi_{n})>v_{\infty}(\xi_{n-1})>\cdots>v_{\infty}(\xi_{1}),\label{f2114}\end{align}
where the equality holds if and only if $n=m$ and $v_{\infty}(\bm{j})=v_{0}q^{m}.$
If $n\geq m+1,$ then we have
\begin{equation}\begin{split}
&v_{\infty}(\xi_{-n})>v_{\infty}(\xi_{-(n-1)})>\cdots>v_{\infty}(\xi_{-(n-m+1)})\\
\geq&v_{\infty}(\xi_{n})>v_{\infty}(\xi_{-(n-m)})\\
\geq&v_{\infty}(\xi_{n-1})>v_{\infty}(\xi_{-(n-m-1)})\\
\geq& \cdots\\ 
\geq& v_{\infty}(\xi_{m+1})>v_{\infty}(\xi_{-1})\\
\geq& v_{\infty}(\xi_{m})>v_{\infty}(\xi_{m-1})>\cdots>v_{\infty}(\xi_{1}),
\end{split}\label{f2115}
\end{equation}
where each equality holds if and only if $v_{\infty}(\bm{j})=v_{0}q^{m}.$
A root of $\phi_{t^{n}}(X)$ having valuation $v_{\infty}(\xi_{i})$ for some $i\in \mathcal{I}_{n}$ is an $\Bbb{F}_{q}$-linear combination of $\xi_{i'}$ such that 
\begin{itemize}
\item
each $\xi_{i'}$ has valuation equal to or larger than $v_{\infty}(\xi_{i});$
\item
at least one of the coefficients of $\xi_{i'}$ with $v_{\infty}(\xi_{i'})=v_{\infty}(\xi_{i})$ is nonzero.
\end{itemize}
\end{proposition}

\begin{proof}
The inequalities (\ref{f2114}) and (\ref{f2115}) follow from 

\begin{enumerate}[(1)]
\item
$v_{\infty}(\xi_{-(i+1)})>v_{\infty}(\xi_{-i})\text{ and }v_{\infty}(\xi_{i+1})>v_{\infty}(\xi_{i})\text{ for }i\geq 1;$
\item
$v_{\infty}(\xi_{-1})\geq v_{\infty}(\xi_{n}),$ where the equality holds if and only if $n=m$ and $v_{\infty}(\bm{j})=v_{0}q^{m}$;
\item
$v_{\infty} (\xi_{-(i+1)}) \geq  v_{\infty} (\xi_{m+i}) > v_{\infty} (\xi_{-i})$ for $i=1,\ldots,n-m,$ where the equality holds if and only if $v_{\infty}(\bm{j})=v_{0}q^{m}.$
\end{enumerate}
\noindent These inequalities follow from Lemma~\ref{l2111}, e.g., the left and the right inequalities of (3) are equivalent to $v_{\infty}(\bm{j})\leq v_{0}q^{m}$ and $v_{\infty}(\bm{j})> v_{0}q^{m+1},$ respectively. 

Let us show that $\phi[t^{n}]$ is generated by $\xi_{i}$ for $i\in \mathcal{I}_{n}.$ 
The case $n=1$ is clear.
Assume that $\xi_{i}$ for $i\in \mathcal{I}_{n-1}$ form a basis of $\phi[t^{n-1}].$
It suffices to show that each $x\in \phi[t^{n}]$ is an $\Bbb{F}_{q}$-linear combination of $\xi_{i}$ for $i\in \mathcal{I}_{n}.$
Indeed, there are $u_{i}$ for $i\in \mathcal{I}_{n-1}$ such that $\phi_{t}(x)=\sum_{i\in \mathcal{I}_{n-1}}u_{i}\cdot \xi_{i}$ since $\phi_{t}(x)\in \phi[t^{n-1}].$
Then we have
\[\phi_{t}\bigg(x-\sum_{i=1}^{n-1}u_{-i}\cdot \xi_{-(i+1)}-\sum_{i=1}^{n-1}u_{i}\cdot \xi_{i+1}\bigg)=0.\]
Hence there exist $u,$ $u'\in\Bbb{F}_{q}$ such that 
\[x-\sum_{i=1}^{n-1}u_{-i}\cdot \xi_{-(i+1)}-\sum_{i=1}^{n-1}u_{i}\cdot \xi_{i+1}=u\cdot \xi_{-1}+u'\cdot \xi_{1}.\]
\end{proof}

\begin{remark}
Assume $v_{\infty}(\bm{j})\neq 0$ and  $v_{\infty}(\bm{j})\neq v_{0}q^{m}$ for any integer $m\geq 2$ in this remark. 
We can determine the valuations of the coefficients of $\phi_{t^{n}}(X)$ in terms of $v_{0},$ $v_{1},$ and $v_{2}.$
Here we describe how to determine the valuations for $n=2.$ 
We have
\[\phi_{t^{2}}(X)=t^{2}X+(a_{1}t+a_{1}t^{q})X^{q}+(a_{2}t+a_{1}^{q+1}+a_{2}t^{q^{2}})X^{q^{2}}+(a_{1}a_{2}^{q}+a_{2}a_{1}^{q^{2}})X^{q^{3}}+a_{2}^{q^{2}+1}X^{q^{4}}.\]
We have by comparing the valuations of the summands of each coefficient \begin{align}
v_{\infty}(t^{2})&=2v_{0},\nonumber\\
v_{\infty}(a_{1}t+a_{1}t^{q})&=v_{\infty}(a_{1}t^{q})=v_{1}+v_{0}q,\nonumber\\
v_{\infty}(a_{2}t+a_{1}^{q+1}+a_{2}t^{q^{2}})&=\begin{cases}
v_{\infty}(a_{1}^{q+1})=v_{1}(q+1)&v_{\infty}(\bm{j})<v_{0}q^{2};\\
v_{\infty}(a_{2}t^{q^{2}})=v_{2}+v_{0}q^{2} & v_{\infty}(\bm{j})>v_{0}q^{2},
\end{cases}\nonumber\\
v_{\infty}(a_{1}a_{2}^{q}+a_{2}a_{1}^{q^{2}})&=\begin{cases}v_{\infty}(a_{2}a_{1}^{q^{2}})=v_{2}+v_{1}q^{2}&v_{\infty}(\bm{j})<0;\\
v_{\infty}(a_{1}a_{2}^{q})=v_{1}+v_{2}q&v_{\infty}(\bm{j})>0,
\end{cases}\nonumber\\
v_{\infty}(a_{2}^{q^{2}+1})&=v_{2}(q^{2}+1).\nonumber
\end{align}
The Newton polygon of $\phi_{t^{2}}(X)$ can be determined from these data by comparing the slopes of all possible segments. 
The valuations of the coefficients of $\phi_{t^{n}}(X)$ for general $n$ can be determined in the same way as above.
However, the calculations are complicated.
\end{remark}

Let us look at the case where $v_{\infty}(\bm{j})\in [v_{0}q,+\infty).$
By (\ref{f2111}), the Newton polygon of $\phi_{t}(X)$ is $P_{0}P_{2}$ having exactly one segment and all nonzero roots of $\phi_{t}(X)$ have valuation
$-\frac{v_{2}-v_{0}}{q^{2}-1}.$
Let $\xi_{-1}$ and $\xi_{1}$ be elements of $\phi[t]$ such that they form a basis. 
For $n\geq 1,$ denote by $\xi_{-(n+1)}$ (resp.\ $\xi_{n+1}$) a root of $\phi_{t}(X)=\xi_{-n}$ (resp.\ $\phi_{t}(X)=\xi_{n}$) such that $v_{\infty}(\xi_{-(n+1)})$ (resp.\ $v_{\infty}(\xi_{n+1})$) is the largest among valuations of all roots of $\phi_{t}(X)=\xi_{-n}$ (resp.\ $\phi_{t}(X)=\xi_{n}$).
There are claims similar to Lemma~\ref{l2111} and Proposition~\ref{p2111}. 
\begin{proposition}\label{p2112}
Assume $v_{\infty}(\bm{j})\in [v_{0}q,+\infty).$
\begin{enumerate}[\rm{(}1)] 
\item
For $n\geq 1,$ we have 
\[v_{\infty}(\xi_{-n})=v_{\infty}(\xi_{n})=-\left(v_{0}(n-1)+\dfrac{v_{2}-v_{0}}{q^{2}-1}\right).\]
Put $Q_{n}=(0,v_{\infty}(\xi_{-n}))=(0,v_{\infty}(\xi_{n})).$
Then the Newton polygons of $\phi_{t}(X)-\xi_{-n}$ and $\phi_{t}(X)-\xi_{n}$ are $Q_{n}P_{0}P_{2}$ having exactly two segments. 

\item
The roots $\xi_{i}$ for $i\in \mathcal{I}_{n}$ form a basis of $\phi[t^{n}].$
For $n\geq 1,$ we have 
\[v_{\infty}(\xi_{-n}) = v_{\infty}(\xi_{n}) >  v_{\infty}(\xi_{-(n-1)}) = v_{\infty}(\xi_{n-1}) > \cdots > v_{\infty}(\xi_{-1}) = v_{\infty}(\xi_{1}).\]
A root of $\phi_{t^{n}}(X)$ having valuation $v_{\infty}(\xi_{i})$ for some $i\in \mathcal{I}_{n}$ is an $\Bbb{F}_{q}$-linear combination of $\xi_{i'}$ such that 
\begin{itemize}
\item 
each $\xi_{i'}$ has valuation equal to or larger than $v_{\infty}(\xi_{i});$ 
\item
at least one of the coefficients of $\xi_{i'}$ with $v_{\infty}(\xi_{i'})=v_{\infty}(\xi_{i})$ is nonzero.
\end{itemize}
\end{enumerate}
\end{proposition}

\begin{proof}
We prove (1) by induction on $n.$ 
We first check the Newton polygons of $\phi_{t}(X)-\xi_{\pm 1}$ (i.e., the polynomials $\phi_{t}(X)-\xi_{1}$ and $\phi_{t}(X)-\xi_{-1}$) are $Q_{1}P_{0}P_{2}$ having exactly two segments.
We have 
\[\mu(Q_{1},P_{0})=v_{0}+\frac{v_{2}-v_{0}}{q^{2}-1},\,\,\mu(Q_{1},P_{1})=\frac{v_{1}+\frac{v_{2}-v_{0}}{q^{2}-1}}{q-0},\,\,\mu(Q_{1},P_{2})=\frac{v_{2}+\frac{v_{2}-v_{0}}{q^{2}-1}}{q^{2}}.\]
We calculate 
\begin{align}
\mu(Q_{1},P_{0})-\mu(Q_{1},P_{1})&=\frac{-v_{\infty}(\bm{j})+v_{0}(q^{2}+q-1)}{(q+1)q}<0,\nonumber\\
\mu(Q_{1},P_{0})-\mu(Q_{1},P_{2})&=\frac{v_{0}(q^{2}-1)}{q^{2}}<0.\nonumber
\end{align}
Then $P_{0}$ is a vertex of the Newton polygon.
The Newton polygons have the desired form since the Newton polygon of $\phi_{t}(X)$ is $P_{0}P_{2}.$

Assume (1) for $n-1.$
Then the valuation of $\xi_{\pm n}$ is $-\mu(Q_{n},P_{0})=-(v_{0}-v_{\infty}(\xi_{\pm (n-1)})).$ 
We show that the Newton polygons of $\phi_{t}(X)-\xi_{\pm n}$ are $Q_{n}P_{0}P_{2}$ having exactly two segments.
We have 
\[\mu(Q_{n},P_{0})=v_{0}-v_{\infty}(\xi_{\pm n}),\,\,\mu(Q_{n},P_{1})=\frac{v_{1}-v_{\infty}(\xi_{\pm n})}{q-0},\,\,\mu(Q_{n},P_{2})=\frac{v_{2}-v_{\infty}(\xi_{\pm n})}{q^{2}-0}.\]
The Newton polygons are determined by the inequalities
\begin{align}
\mu(Q_{n},P_{0})-\mu(Q_{n},P_{1})&=\frac{-v_{\infty}(\bm{j})+v_{0}((q-1)(q+1)n+q)}{(q+1)q}<0,\nonumber\\
\mu(Q_{n},P_{0})-\mu(Q_{n},P_{2})&=\frac{v_{0}n(q^{2}-1)}{q^{2}}<0.\nonumber
\end{align}

The proof of (2) is similar to that of Proposition~\ref{p2111}.
\end{proof}

\subsection{Finite primes}\label{s212}
Let us denote $v$ by $v_{f}$ to emphasize that we are working with a finite prime.
Let $v_{0},$ $v_{1},$ and $v_{2}$ denote respectively the valuation of $t,$ $a_{1},$ and $a_{2},$ which appear in  $\phi_{t}(X)=tX+a_{1}X^{q}+a_{2}X^{q^{2}}.$
Since we assume $v\nmid t,$ we now have $v_{f}(t)=v_{0}=0.$
If $v_{f}(\bm{j})<0,$ then a calculation similar to that in Section~\ref{s211} shows that the Newton polygon of $\phi_{t}(X)$ has exactly two segments.
In this case, let $\xi_{-1}$ (resp.\ $\xi_{1}$) denote a nonzero root with larger valuation (resp.\ with smaller valuation).
For $n\geq 1,$ let $\xi_{-(n+1)}$ (resp.\ $\xi_{n+1}$) denote a root of $\phi_{t}(X)=\xi_{-n}$ (resp.\ $\phi_{t}(X)=\xi_{n}$) with the largest valuation. 

If $v_{f}(\bm{j})\geq 0,$ then the Newton polygon of $\phi_{t}(X)$ has exactly one segment.
Let $\xi_{-1}$ and $\xi_{1}$ be elements of $\phi[t]$ such that they form a basis. 
For $n\geq 1,$ let $\xi_{-(n+1)}$ (resp.\ $\xi_{n+1}$) denote a root of $\phi_{t}(X)=\xi_{-n}$ (resp.\ $\phi_{t}(X)=\xi_{n}$).
We have the following claims similar to Lemma~\ref{l2111}, Proposition~\ref{p2111}, and Proposition~\ref{p2112}.
\begin{proposition}\label{p2121}
Let $n$ be a positive integer.
Put $P_{0}=(1,v_{0}),$ $P_{1}=(q,v_{1}),$ and $P_{2}=(q^{2},v_{2}).$
\begin{enumerate}[\rm{(}1)]
\item If $v_{f}(\bm{j})<0,$ we have 
\begin{align}v_{f}(\xi_{n})&=-\dfrac{v_{2}+v_{1}(q^{n}-q-1)}{(q-1)q^{n}},\nonumber\\
v_{f}(\xi_{-n})&=-\dfrac{v_{1}}{q-1}.\nonumber\end{align}
Put $Q_{n}=(0,v_{f}(\xi_{n}))$ and $Q_{-n}=(0,v_{f}(\xi_{-n})).$ 
Then the Newton polygon of $\phi_{t}(X)-\xi_{n}$ \text{\rm{(}}resp.\ $\phi_{t}(X)-\xi_{-n}$\text{\rm{)}} is $Q_{n}P_{1}P_{2}$ \text{\rm{(}}resp.\  $Q_{-n}P_{1}P_{2}$\text{\rm{)}} having exactly two segments.

Moreover, $\xi_{i}$ for $i\in \mathcal{I}_{n}$ form a basis of $\phi[t^{n}].$
We have
\[v_{f}(\xi_{-n})=v_{f}(\xi_{-(n-1)})=\cdots=v_{f}(\xi_{-1})>v_{f}(\xi_{n})>v_{f}(\xi_{n-1})>\cdots>v_{f}(\xi_{1}).\]
A root of $\phi_{t^{n}}(X)$ having the valuation $v_{f}(\xi_{i})$ for some $i\in \mathcal{I}_{n}$ is an $\Bbb{F}_{q}$-linear combination of  $\xi_{i'}$ such that 
\begin{itemize}
\item
each $\xi_{i'}$ has valuation equal to or larger than $v_{f}(\xi_{i});$ 
\item
at least one of the coefficients of $\xi_{i'}$ with $v_{f}(\xi_{i'})=v_{f}(\xi_{i})$ is nonzero.\end{itemize}

\item If $v_{f}(\bm{j})\geq 0,$ we have \[v_{f}(\xi_{-n})=v_{f}(\xi_{n})=-\dfrac{v_{2}}{q^{2}-1}.\]
Put $Q=(0,v_{f}(\xi_{-n}))=(0,v_{f}(\xi_{n})).$ Then the Newton polygons of $\phi_{t}(X)-\xi_{-n}$ and $\phi_{t}(X)-\xi_{n}$ are $Q_{n}P_{2}$ having exactly one segment.
The roots $\xi_{i}$ for $i\in \mathcal{I}_{n}$ form a basis of $\phi[t^{n}].$
\end{enumerate}
\end{proposition}

\begin{proof}
The claims for the valuations and  for the Newton polygons are proved by induction on $n.$ 
For (1), we first show that the Newton polygons of $\phi_{t}(X)-\xi_{\pm 1}$ are $Q_{\pm 1}P_{1}P_{2}$ having exactly two segments.
We have 
\[\mu(Q_{\pm 1},P_{0})=-v_{f}(\xi_{\pm 1}),\,\,\mu(Q_{\pm 1},P_{1})=\frac{v_{1}-v_{f}(\xi_{\pm 1})}{q-0}
,\,\,\mu(Q_{\pm 1},P_{2})=\frac{v_{2}-v_{f}(\xi_{\pm 1})}{q^{2}-0}.\]
The Newton polygon of $\phi_{t}(X)-\xi_{1}$ is determined by the inequalities
\begin{align}
\mu(Q_{1},P_{0})-\mu(Q_{1},P_{1})&=\frac{-v_{f}(\bm{j})}{q^{2}}>0,\nonumber\\
\mu(Q_{1},P_{1})-\mu(Q_{1},P_{2})&=\frac{v_{f}(\bm{j})(q-1)}{q^{3}}<0\nonumber.
\end{align}
The Newton polygon of $\phi_{t}(X)-\xi_{-1}$ is determined by the equality  $\mu(Q_{-1},P_{0})-\mu(Q_{-1},P_{1})=0,$
and the inequality \[\mu(Q_{-1},P_{1})-\mu(Q_{-1},P_{2})=\frac{v_{f}(\bm{j})}{q^{2}}<0.\]

Assume that (1) for $n-1$ is valid.
The valuations of $\xi_{\pm n}$ are  \begin{align}-\mu(Q_{n-1},P_{1})&=-\frac{v_{1}-v_{f}(\xi_{n-1})}{q-0}=-\frac{v_{2}+v_{1}(q^{n}-q-1)}{(q-1)q^{n}},\nonumber\\
-\mu(Q_{-(n-1)},P_{1})&=-\frac{v_{1}-v_{f}(\xi_{-(n-1)})}{q-0}=-\frac{v_{1}}{q-1}.\nonumber
\end{align}
We show that the Newton polygons of $\phi_{t}(X)-\xi_{\pm n}$ are $Q_{\pm n}P_{1}P_{2}$ having exactly two segments.
We have 
\[\mu(Q_{\pm n},P_{0})=-v_{f}(\xi_{\pm n}),\,\,\mu(Q_{\pm n},P_{1})=\frac{v_{1}-v_{f}(\xi_{\pm n})}{q-0},\,\,\mu(Q_{\pm n},P_{2})=\frac{v_{2}-v_{f}(\xi_{\pm n})}{q^{2}-0}.\]
The Newton polygon of $\phi_{t}(X)-\xi_{n}$ is determined by the inequalities
\begin{align}
\mu(Q_{n},P_{0})-\mu(Q_{n},P_{1})&=\frac{-v_{f}(\bm{j})}{q^{n+1}}>0,\nonumber\\
\mu(Q_{n},P_{1})-\mu(Q_{n},P_{2})&=\frac{v_{f}(\bm{j})(q^{n}-1)}{q^{n+2}}<0.\nonumber
\end{align}
The Newton polygon of $\phi_{t}(X)-\xi_{-n}$ is determined by the equality $\mu(Q_{-n},P_{0})-\mu(Q_{-n},P_{1})=0,$
and the inequality \[\mu(Q_{-n},P_{1})-\mu(Q_{-n},P_{2})=\frac{v_{f}(\bm{j})}{q^{2}}<0.\]
The proof of the rest of (1) is similar to that of Proposition~\ref{p2111}.

For (2), we first check the Newton polygons of $\phi_{t}(X)-\xi_{\pm 1}$ are $QP_{2}$ having exactly one segment.
We have 
\[\mu(Q,P_{0})=-v_{f}(\xi_{\pm 1}),\,\,\mu(Q,P_{1})=\frac{v_{1}-v_{f}(\xi_{\pm 1})}{q-0},\,\,\mu(Q,P_{2})=\frac{v_{2}-v_{f}(\xi_{\pm 1})}{q^{2}-0}.\]
The Newton polygons of $\phi_{t}(X)-\xi_{\pm 1}$ are determined by the equality $\mu(Q,P_{0})-\mu(Q,P_{2})=0,$ and the inequality
\[\mu(Q,P_{1})-\mu(Q,P_{2})=\frac{v_{f}(\bm{j})}{(q+1)q}\geq 0.\]
Assume that (2) for $n-1$ is valid. 
The valuations of $\xi_{\pm n}$ are calculated by $-\mu(Q,P_{2}).$
Since $v_{f}(\xi_{\pm n})=v_{f}(\xi_{\pm 1}),$ the Newton polygons of $\phi_{t}(X)-\xi_{\pm n}$ are the same as those of $\phi_{t}(X)-\xi_{\pm 1}.$ 
The proof of the rest of (2) is similar to that of Proposition~\ref{p2111}.
\end{proof}

\section{The Herbrand $\psi$-functions and ramification subgroups for extensions by division points}\label{s22}
\subsection{The Herbrand $\psi$-functions}
Let us recall the definition of the (Herbrand) $\psi$-function $\psi_{L/K}$ for a finite Galois extension $L/K$ of a complete valuation field.
Let $G^{y}$ denote the $y$-th upper ramification subgroup of the Galois group $G(L/K)$ of $L/K.$
By the $\psi$-function of $L/K,$ we mean the real-valued function on the interval $[0,+\infty)$ defined as 
\[\psi_{L/K}(y)=\int_{0}^{y}\frac{\# G^{0}}{\# G^{r}}dr.\]
We extend $\psi_{L/K}$ to $[-1,+\infty)$ by letting $\psi_{L/K}(y)=y$ if $-1\leq y\leq 0.$
Then $\psi_{L/K}$ is a continuous and piecewise linear function on $[-1,+\infty).$
If $\psi_{L/K}$ is linear on some interval $[a,b]\subset [-1,\infty),$ then we have $G^{b}=G^{y}=G_{\psi_{L/K}(y)}$ for $y\in (a,b].$
By the ramification break of $L/K,$ we mean the real number $\psi_{L/K}(y),$ where $y\geq 0$ is the maximal real number such that $G^{y}\neq 1.$ 
By the wild ramification subgroup of $L/K,$ we mean the first lower ramification subgroup $G_{1},$ which is equal to the union of $G^{y}$ for $y>0.$

\begin{lemma}[\text{\rm{see e.g., \cite[Chapter~III,\,(3.3)]{FV}}}]\label{l221}
Let $L/M$ and $M/K$ be finite Galois extensions of complete valuation fields. Then 
\[\psi_{L/K}=\psi_{L/M}\circ \psi_{M/K}.\]
\end{lemma}

The extension generated by the roots of the polynomial $X^{q}-X-\alpha$ (for some $\alpha$ in the base field) plays a key role in this section.
To obtain its $\psi$-function, we will need the following fact. It is a slight generalization of the function field case of \cite[Chapter~III,\,Proposition~2.5]{FV}.

\begin{proposition}\label{p221}
Let $K$ be a complete discrete valuation field of characteristic $p>0.$
Assume that $K$ contains $\Bbb{F}_{q},$ where $q=p^{s}$ for some $s.$
Let $v_{K}$ denote the normalized valuation.
Let $\alpha$ be an element of $K$ and $\lambda$ a root of $X^{q}-X-\alpha.$
If $v_{K}(\alpha)<0$ and $p\nmid v_{K}(\alpha),$ then the field extension $K(\lambda)/K$ is totally ramified of degree $q$ with Galois group isomorphic to $(\Bbb{Z}/p\Bbb{Z})^{s}.$ Its ramification break is $r\coloneqq -v_{K}(\alpha)$ and the $\psi$-function is  
\[\psi_{K(\lambda)/K}(y)=\begin{cases}
y&-1\leq y\leq r;\\
qy-(q-1)r &r\leq y.\end{cases}
\]
\end{proposition}
\begin{proof}
Let $L$ be the splitting field of $X^{q}-X-\alpha.$ 
Then $L=K(\lambda)$ since the difference of any two roots of $X^{q}-X-\alpha$ belongs to $\Bbb{F}_{q}\subset K.$ 
Since $p\nmid v_{K}(\alpha)$ and $ v_{K}(\alpha)<0,$ we have by the Newton polygon  $v_{K}(\lambda)=-v_{K}(\alpha)/q,$ which implies that  the ramification index of $L/K$ is $q.$ 
Thus $L/K$ is a totally ramified Galois extension of degree $q$.  For any $\sigma\in G(L/K),$ we have $\sigma(\lambda)-\lambda\in \Bbb{F}_{q}.$
The element $\sigma$ is determined by $\sigma(\lambda)$ since $\lambda$ generates $L/K.$ Thus $\sigma \mapsto \sigma(\lambda)-\lambda$ defines an isomorphism $G(L/K)\to \Bbb{F}_{q}\cong (\Bbb{Z}/p\Bbb{Z})^{s}.$ The ramification break of $L/K$ is $v_{L}(\sigma(\lambda)\lambda^{-1}-1)=v_{L}(\sigma(\lambda)-\lambda)-v_{L}(\lambda)=r$ for a nontrivial $\sigma.$ 
The $\psi$-function $\psi_{L/K}$ is as the proposition describes.
\end{proof}

\subsection{Infinite primes}\label{s221}
From now on, let us resume the notation in Section~\ref{s21}.
Let $K$ denote the completion of $F$ at an infinite prime or a finite prime $v$ which does not divide $t.$
Let $\phi$ be a rank $2$ Drinfeld $\Bbb{F}_{q}[t]$-module over $K.$
Let $K_{n}$ denote $K(\phi[t^{n}])$ for each $n\geq 1.$
In the rest of this section, based on the results in Section~\ref{s21}, we are concerned with the $\psi$-function of the extension $K_{n}/K.$ 
Let us first consider the case where $v=v_{\infty}$ is infinite.
Recall $v_0 = v_\infty(t) < 0.$

\begin{lemma}\label{l2211}
Assume $v_{\infty}(\bm{j})<v_{0}q.$ 
Let $m$ be the integer satisfying $v_{\infty}(\bm{j})\in (v_{0}q^{m+1},v_{0}q^{m}].$ 
We have $\xi_{-n-1}\in K_{n}$ for $n\geq 1$ and $\xi_{n+1}\in K_{n}$ for $n\geq m,$ i.e., 
\[K_{n}(\xi_{-n-1})=K_{n}\text{ for }n\geq 1\text{ and }K_{n}(\xi_{n+1})=K_{n}\text{ for }n\geq m.\]
\end{lemma}

\begin{proof}
By Lemma~\ref{l2111}~(1), for $n\geq m,$ the Newton polygon of $\phi_{t}(X)-\xi_{n}$ is $Q_{n}P_{0}P_{1}P_{2}$ having exactly three segments, where $Q_{n}$ and $P_{0}$ have horizontal coordinates $0$ and $1,$ respectively. 
We see that $\xi_{n+1}$ is the only root of $\phi_{t}(X)-\xi_{n}$ whose valuation is $-\mu(Q_{n},P_{0}).$ 
Thus due to \cite[Chapter~II,\,Proposition~6.4]{Neu}, we know $(X-\xi_{n+1})\in K_{n}[X]$ and thus $\xi_{n+1}\in K_{n}.$
One can show $\xi_{-n-1}\in K_{n}$ for $n\geq 1$ in the same way.
\end{proof}

We now study the extension $K_{n+1}/K_{n}$ for $n\leq m-1$ and need the following technical lemma.
Note that we have $\prod_{u\in \Bbb{F}_{q}}(X-u\cdot \xi_{-1})\in K[X]$ by \cite[Chapter~II,\,Proposition~6.4]{Neu}.

\begin{lemma}\label{l2212}
Assume $v_{\infty}(\bm{j})<v_{0}q.$
Put  
\begin{align}\eta(X)\coloneqq a_{1}X^{q}+tX.\label{f22110}
\end{align}
Let $K'$ denote the extension of $K$ generated by the roots of $\eta(X).$
Then $K'=K(\xi_{-1}).$
\end{lemma}

\begin{proof}
Our first goal is to show $K'\subset K(\xi_{-1}).$
Let $x_{i}'$ for $i=1, \ldots, q-1$ denote all nonzero roots of $\eta(X).$
Let $|-|$ denote the absolute value on $K^{\text{sep}}$ given by $q^{-v_{\infty}(-)}.$
Note that $K' = K(x_{i}')$ for all $i$ as $K$ contains all $(q-1)$-st roots of unity.
Due to Krasner's lemma (see \cite[p.\,152]{Neu}), it suffices to show that there exists some $i$ such that $|\xi_{-1}-x_{i}'|<|x_{i}'-x_{j}'|$ for any $j \neq i.$
Since $x_{i}'-x_{j}'$ for $i\neq j$ is a nonzero root of $\eta(X),$ we have $|x_{i}'-x_{j}'|=q^{\frac{v_{1}-v_{0}}{q-1}}.$ 
Consider 
\begin{align}\eta(\xi_{-1})=a_{1}\cdot \xi_{-1}\prod_{i=1}^{q-1}(\xi_{-1}-x_{i}')=a_{1}\cdot \xi_{-1}^{q}+t\cdot \xi_{-1}=-a_{2}\cdot \xi_{-1}^{q^{2}}.\nonumber\end{align}
Since 
\begin{align}v_{\infty}(a_{1}^{-1}a_{2}\cdot \xi_{-1}^{q^{2}-1})&=-v_{1}+v_{2}-(v_{1}-v_{0})(q+1)\nonumber\\
&=-v_{\infty}(\bm{j})+v_{0}q-(v_{1}-v_{0}),\nonumber
\end{align}
we have 
\[\left|\prod_{i=1}^{q-1}(\xi_{-1}-x_{i}')\right| = \left|a_{1}^{-1}a_{2}\cdot \xi_{-1}^{q^{2}-1}\right| = q^{v_{\infty}(\bm{j})-v_{0}q+(v_{1}-v_{0})}.\]
As $v_{\infty}(\bm{j}) < v_{0}q,$ there exists some $x_{i}'$ such that
\[|\xi_{-1} - x_{i}'| \leq  q^{\frac{v_{\infty}(\bm{j})-v_{0}q}{q-1}+\frac{v_{1}-v_{0}}{q-1}}<q^{\frac{v_{1}-v_{0}}{q-1}}=|x_{i}'-x_{j}'|,\]
as desired.

Conversely, the conjugates of $\xi_{-1}$ are of the form $u\cdot \xi_{-1}$ for $u\in \Bbb{F}_{q}^{\times}$ and hence $|\xi_{-1}-x_{i}'|<|\xi_{-1}-u\cdot \xi_{-1}|=q^{\frac{v_{1}-v_{0}}{q-1}}$ for $u \in \Bbb{F}_{q}\setminus \{0,1\}.$ 
Apply Krasner's lemma again and $K(\xi_{-1})\subset K'$ follows.
\end{proof}

We need the following modification of certain polynomials.
Assume $v_{\infty}(\bm{j})<v_{0}q.$ 
Let $m$ be the integer satisfying $v_{\infty}(\bm{j})\in (v_{0}q^{m+1},v_{0}q^{m}].$
We can choose an element $b$ in $K^{\prime \times}$ with $v_{\infty}(b)=\frac{v_{1}-v_{0}}{q-1}$
to modify $\phi_{t}(X),$ $\phi_{t}(X)-\xi_{n},$ $\eta(X),$ and $\eta(X)-\xi_{n}$ for $n < m,$ so that the coefficients of $X^{q}$ of the obtained polynomials are $1$:
\begin{align}\label{f2217}
\left. \begin{alignedat}{2}\Phi(X) & = b_{2}X^{q^{2}}+X^{q}+b_{0}X && \coloneqq b'\phi_{t}(X/b), 
\\
\Phi_{n}(X) & = b_{2}X^{q^2}+X^q+b_0X+c_{n} && \coloneqq b'(\phi_{t}(X/b)-\xi_{n}),\\
\Eta(X) & = X^{q}+b_{0}X && \coloneqq b'\eta(X/b), 
\\
\Eta_{n}(X) & = X^{q}+b_{0}X+c_{n} && \coloneqq b'(\eta(X/b)-\xi_{n})  
\end{alignedat}\right.
\end{align}
with $b'=b^{q}/a_{1}.$ 
Then
\[v_{\infty}(c_{n})=\frac{v_{\infty}(\bm{j})-v_{0}q^{n+1}}{(q-1)q^{n}},\,\,v_{\infty}(b_{0})=0,\text{ and }v_{\infty}(b_{2})=-v_{\infty}(\bm{j})+v_{0}q.\]
Notice that all nonzero roots of $\Eta(X)$ have valuation $0$ and belong to $K_{1}.$ 

In \cite[Proposition~3]{KL}, K\"olle and Schmid applied Krasner's lemma to study unramified or tamely ramified extensions of a number field. 
Their proposition roughly claims that two polynomials yield the same field extension if their Newton polygons are the same. 
The following lemma is the analogue for certain wildly ramified extensions. 

\begin{lemma}\label{l2213}
    Assume $v_{\infty}(\bm{j})\leq v_{0}q^{2}.$ Let $m\geq 2$ be the integer such that $v_{\infty}(\bm{j})\in (v_{0}q^{m+1},v_{0}q^{m}].$
    If $1\leq n< m,$ then any root $\xi'$ of the polynomial $\eta(X)-\xi_{n}$
    satisfies
    \[K_{n}(\xi')= K_{n+1}.\]
\end{lemma}

Due to this lemma, the extension $K_{n+1}/K_{n}$ for $n=1, \ldots, m-1$ is generated by a root of $a_{1}X^{q}+tX-\xi_{n}\in K_{n}[X].$

\begin{proof}
The roots of $\phi_{t}(X)-\xi_{n}$ whose valuations equal $v_{\infty}(\xi_{n+1})$ are $\xi_{n+1}+u\cdot \xi_{-1}$ for all $u\in \Bbb{F}_{q}.$ 
Due to \cite[Chapter~II,\,Proposition~6.4]{Neu}, we have \[\prod_{u\in\Bbb{F}_{q}}(X-\xi_{n+1}-u\cdot \xi_{-1})\in K_{n}[X].\]
(Note that this polynomial and $\eta(X) - \xi_{n}$ have the same Newton polygon.)
Since $K_{n+1}=K_{n}(\xi_{n+1})$ by Lemma~\ref{l2211}, $K_{n}(\xi_{n+1})$ is the splitting field of this polynomial. 
Then we are reduced to showing $K_{n}(\xi')=K_{n}(\xi_{n+1}).$
Let $x=b\xi_{n+1}.$
Let $x_{i}'$ for $i=1, \ldots, q$ denote the roots of $\Eta_{n}(X).$
The difference $b\xi'-x_{i}'$ belongs to $K' \subset K_{1},$ as it is a root of $\Eta(X).$ 
Since $K_{n}(x_{i}') = K_{n}(b\xi') = K_{n}(\xi'),$ it suffices to show $K_{n}(x_{i}') = K_{n}(x)$ for some $i.$

If there exists $i$ such that $|x-x_{i}'|<|x_{i}'-x_{j}'|$ for all $j\neq i,$ then we apply Krasner's lemma and obtain $K_{n}(x_{i}')\subset K_{n}(x).$
We have $|x_{i}'-x_{j}'|=1$ for all $i\neq j$ since $x_{i}'-x_{j}'$ is a nonzero root of $\Eta(X).$
It suffices to find a suitable root $x_{i}'$ of $\Eta_{n}(X)$ such that $|x-x_{i}'|<1.$
To know $|x-x_{i}'|,$ we consider the valuation of
\[\Eta_{n}(x)=\prod_{i=1}^{q}(x-x_{i}')=x^{q}+b_{0}x+c_{n}=-b_{2}x^{q^{2}},\]
where the rightmost equality comes from $\Phi_{n}(x)=0.$
We have
\begin{equation}\begin{split}v_{\infty}(b_{2}x^{q^{2}})&=(-v_{\infty}(\bm{j})+v_{0}q)+\frac{v_{\infty}(\bm{j})-v_{0}q^{n+1}}{(q-1)q^{n-1}}\\
&=-v_{\infty}(\bm{j})\left(1-\frac{1}{(q-1)q^{n-1}}\right)-v_{0}\frac{q}{q-1}.\end{split}\label{f312}\end{equation} 
Since $v_{\infty}(\bm{j})\leq v_{0}q^{2},$ we have 
\[v_{\infty}(b_{2}x^{q^{2}})\geq -v_{0}q^{2}\left(1-\frac{1}{q-1}\right)-v_{0}\frac{q}{q-1}>0,\] 
which induces
\[\left|\prod_{i=1}^{q}(x-x_{i}')\right|<1.\]
Hence there exists some $x_{i}'$ such that $|x-x_{i}'|<1.$

Conversely, all conjugates of $x=b\xi_{n+1}$ are of the form $b(\xi_{n+1}+u\cdot \xi_{-1})$ for $u\in \Bbb{F}_{q}.$ For $u \in \Bbb{F}_{q}^{\times},$ then $|b\xi_{n+1}-b(\xi_{n+1}+u\cdot \xi_{-1})|=1.$ 
Since $|x-x_{i}'|<1,$ Krasner's lemma implies $K_{n}(x)\subset K_{n}(x_{i}').$
\end{proof}

We must look into $K_{1}/K.$ 
Recall $K'=K(\xi_{-1})$ by Lemma~\ref{l2212}.
Let us look into the extension $K_{1}/K'.$ 
Due to $v_{\infty}(\xi_{1})=-\frac{v_{2}-v_{1}}{(q-1)q},$ it is natural to ask whether we can write $K_{1}/K$ as a compositum of extensions of degree dividing $q-1$ and of degree dividing $q.$

\begin{lemma}\label{l2214}
Assume $v_{\infty}(\bm{j})<v_{0}q.$ 
Let $K_{1}'$ be the splitting field over $K'$ of the degree $q$ polynomial 
\[\widetilde{\Eta}(X)=X^{q}-\beta X-\beta\in K'[X]\text{ with }\beta=\frac{\xi_{-1}^{q^{2}-1}a_{2}}{t}.\]
Then
\begin{itemize} 
\item
$K_{1}'/K'$ is a subextension of $K_{1}/K';$
\item
$K_{1}/K_{1}'$ is a compositum of Kummer extensions.
\end{itemize}
\end{lemma}

\begin{proof}
We put $b=\xi_{-1}^{-1}$ in the definition of $\Phi(X)$ in (\ref{f2217}) so that any element in $\Bbb{F}_{q}$ is a root of $\Phi(X).$ 
We have that $X^{q}-X=\prod_{u\in \Bbb{F}_{q}}(X-u)$ divides $\Phi(X).$
Let $\Theta(X)$ be the polynomial such that $\Theta(X)(X^{q}-X)=\Phi(X).$
Noticing that $-b_{0}-b_{2}=1$ by $\Phi(1)=0$ and $b_{0}/b_{2}=1/\beta,$ we have
\[\Theta(X)=b_{2}\left(X^{q(q-1)}+X^{(q-1)(q-1)}+\cdots+X^{q-1}
-1/\beta\right),\]
whose roots generate $K_{1}/K'.$

We consider $\overline{\Theta}(X)\coloneqq (X^{q}+X^{q-1}+\cdots+X)-1/\beta$ and the subextension $L$ of $K_{1}/K'$ generated by its roots. 
Then $K_{1}/L$ is generated by the $(q-1)$-st roots of $x$ for all $x\in L$ satisfying $\overline{\Theta}(x)=0,$
which implies that $K_{1}/L$ is a compositum of Kummer extensions. 

We prove $L=K_{1}'$ by showing that $L$ is the splitting field of $\widetilde{\Eta}(X).$
Notice
\[\sum_{i=1}^{q}X^{i}=\frac{X(X^{q}-1)}{X-1}=\frac{X(X-1)^{q}}{X-1}=X(X-1)^{q-1}.\] 
Hence $\overline{\Theta}(X+1)=X^{q}+X^{q-1}-1/\beta.$ 
Then 
\[\widetilde{\Eta}(X)=-\beta\cdot  X^{q}\cdot \overline{\Theta}\left(\frac{1}{X}+1\right).\] 
Thus $L=K_{1}',$ as desired.
\end{proof}

The difference of two roots of $\widetilde{\Eta}(X)$ is a root of $X^{q}-\beta X.$ 
However, it might not be in $K'.$ 
This is why the extension of $K'$ generated by one root of $\widetilde{\Eta}(X)$ might not be Galois in general. 

From now on, assume that the valuation $v_{\infty}(\bm{j})$ of the $j$-invariant is not divisible by $p.$ 
Since $v_{1}-v_{2}=v_{\infty}(\bm{j})-v_{1}q,$ we have $p\nmid v_{1}-v_{2}.$ 
Hence $q$ divides the ramification index of $K_{1}/K$ since $v_{\infty}(\xi_{1})=\frac{v_{1}-v_{2}}{(q-1)q}.$
Let $E$ be the integer such that $Eq$ is the ramification index of $K_{1}/K.$ 

We then apply Proposition~\ref{p221} to obtain the $\psi$-function of $K_{n}/K$ for all $n.$ 
Due to Lemma~\ref{l221}, we first consider the $\psi$-functions of $K_{1}/K$ and of $K_{n+1}/K_{n}$ as follows.

\begin{lemma}\label{l2216}
Assume $v_{\infty}(\bm{j})<v_{0}q$ and $p\nmid v_{\infty}(\bm{j}).$ 
Let $m$ be the integer satisfying $v_{\infty}(\bm{j})\in(v_{0}q^{m+1},v_{0}q^{m}).$ 
\begin{enumerate}[\rm{(}1)]
    \item 
Let $e_{0}\coloneqq e_{K'/K}$ and $e_{1}\coloneqq e_{K_{1}/K_{1}'}$ respectively be the ramification indices of $K'/K$ and $K_{1}/K_{1}'.$ Let $e'$ denote the ramification index of the extension of $K'$ generated by all roots of $X^{q}-\beta X.$ Then we have $E=e_{1}e'e_{0}$ and 
\begin{align}
\psi_{K_{1}/K}(y)=\begin{cases}
y & -1\leq y\leq 0;\\
Ey & 0\leq y\leq \dfrac{r_{1}}{E};\\
qEy-(q-1)r_{1} & \dfrac{r_{1}}{E}\leq y,
\end{cases}\nonumber\end{align}
where $r_{1}\coloneqq E\frac{-v_{\infty}(\bm{j})+v_{0}q}{q-1}.$
    \item 
\begin{enumerate}[\rm{(}i)]
\item
For $1\leq n\leq m,$ the ramification index of $K_{n}/K$ is $Eq^{n};$ 
\item
For $1\leq n\leq m-1,$ we have 
\[\psi_{K_{n+1}/K_{n}}(y)=\begin{cases}
y & -1\leq y\leq r_{n+1};\\
qy-(q-1)r_{n+1} & r_{n+1}\leq y,
\end{cases}\]
where $r_{n+1}\coloneqq E\frac{-v_{\infty}(\bm{j})+v_{0}q^{n+1}}{q-1}.$
\end{enumerate}
\end{enumerate}
\end{lemma}
\begin{proof}
We show (1). Due to Lemma~\ref{l2214}, the extension $K_{1}/K$ is decomposed into the tower
\[\xymatrix@=1.0em{K\ar@{-}[r]&K'\ar@{-}[r]&K_{1}'\ar@{-}[r]&K_{1}.}\]
Due to Lemma~\ref{l2212}, we know that $K'$ is the splitting field of $a_{1}X^{q}+tX$ over $K$ and the extension $K'/K$ is Kummer.
By Lemma~\ref{l2214}, the extension $K_{1}/K_{1}'$ is a compositum of Kummer extensions.
Hence $p\nmid e_{0},$ $p\nmid e_{1},$ and we have
\[\psi_{K'/K}(y)=\begin{cases}
y&-1\leq y\leq 0;\\
e_{0}y&0\leq y,
\end{cases}\text{ and }
\psi_{K_{1}/K_{1}'}(y)=\begin{cases}
y&-1\leq y\leq 0;\\
e_{1}y&0\leq y.
\end{cases}
\]
Then it suffices to know $\psi_{K_{1}'/K'}.$
Let $K''/K'$ denote the Kummer extension generated by the roots of $X^{q}-\beta X.$ 
It is a subextension of $K_{1}'/K'$ because the difference of two roots of $\widetilde{\Eta}(X)$ is a root of $X^{q}-\beta X.$
Its $\psi$-function is
\[\psi_{K''/K'}(y)=\begin{cases}
y&-1\leq y\leq 0;\\
e'y&0\leq y.
\end{cases}\]

Consider the polynomial $\widehat{\Eta}(X)=\omega^{-q}\cdot \widetilde{\Eta}(\omega \cdot X)=X^{q}-X-\omega^{-1}\in K''[X],$ where $\omega$ is a nonzero root of $X^{q}-\beta X.$
Then the extension $K_{1}'/K''$ is generated by one root of $\widehat{\Eta}(X).$ 
Because $p\nmid -v_{K''}(\omega^{-1})$ and $-v_{K''}(\omega^{-1})=e'e_{0}\frac{-v_{\infty}(\bm{j})+v_{0}q}{q-1}>0$ with $v_{K''}=e'e_{0}v_{\infty},$ we can apply Proposition~\ref{p221} to $\widehat{\Eta}(X).$
Thus $K_{1}'/K''$ is a degree $q$ totally ramified Galois extension. 
This implies $E=e_{1}e'e_{0}.$ The $\psi$-function of $K_{1}'/K''$ is 
\[\psi_{K_{1}'/K''}(y)=\begin{cases}
y&-1\leq y\leq v_{K''}(\omega);\\
qy-(q-1)v_{K''}(\omega)&v_{K''}(\omega)\leq y.
\end{cases}\]
Finally, (1) follows from $\psi_{K_{1}/K}(y)=\psi_{K_{1}/K_{1}'}\circ \psi_{K_{1}'/K''}\circ \psi_{K''/K'}\circ \psi_{K'/K}(y).$

We show (i) of (2).  
The case $n=1$ is known.
Assume that (i) is valid for $n.$ 
To show that  $K_{n+1}/K$ has ramification index $Eq^{n+1},$ it suffices to show that the ramification index of $K_{n+1}/K_{n}$ is $q.$ 
Recall that the extension $K_{n+1}/K_{n}$ is generated by a root of $\Eta_{n}(X)=X^{q}+b_{0}X+c_{n}$ in (\ref{f2217}). 
We can take $b_{0}=-1$ in $\Eta_{n}(X)$ by putting $b'=b^{q}/a_{1}$ and $1/b$ to be a nonzero root of $\eta(X).$ 
Because $p\nmid -v_{K_{n}}(c_{n})$ and $-v_{K_{n}}(c_{n})=r_{n+1}>0$ with $v_{K_{n}}=Eq^{n}v_{\infty},$ we can apply Proposition~\ref{p221} to $\Eta_{n}(X).$
Hence $K_{n+1}/K_{n}$ is a degree $q$ totally ramified Galois extension and this shows (i). 
We also know from Proposition~\ref{p221} that 
\[\psi_{K_{n+1}/K_{n}}(y)=\begin{cases}
y & -1\leq y\leq r_{n+1};\\
qy-(q-1)r_{n+1} & r_{n+1}\leq y
\end{cases}\]
and (ii) of (2) follows.
\end{proof}

With the notation and assumptions in this lemma, we have the decomposition of $K_{n}/K$ for $1\leq n\leq m,$ 
\[\xymatrix@=1.0em{K\ar@{-}[r]^{e_{0}}&K'\ar@{-}[r]^{e'}&K''\ar@{-}[r]^{q}&K_{1}'\ar@{-}[r]^{e_{1}}&K_{1}\ar@{-}[r]^{q}&K_{2}\ar@{-}[r]&\cdots\ar@{-}[r]&K_{n-1}\ar@{-}[r]^{\,\,\,\,q}&K_{n}},\]
where each number indicates the ramification index of the corresponding extension.

Due to Lemma~\ref{l221}, we can show by induction that
\begin{lemma}\label{l2215}
Assume $v_{\infty}(\bm{j})<v_{0}q$ and $p\nmid v_{\infty}(\bm{j}).$ 
Let $m$ be the integer satisfying $v_{\infty}(\bm{j})\in(v_{0}q^{m+1},v_{0}q^{m}).$ 
Put $r_{n}=E\frac{-v_{\infty}(\bm{j})+v_{0}q^{n}}{q-1}$ for any positive integer $n$ as in Lemma~\text{\rm{\ref{l2216}}}. 
Then for $n\leq m,$ we have
\begin{equation}\psi_{K_{n}/K}(y)=\begin{cases}
y &-1\leq y\leq 0;\\
Ey & 0\leq y\leq \dfrac{r_{n}}{E};\\
q^{j}Ey-\sum_{i=0}^{j-1}q^{i}(q-1)r_{n-i} & \begin{split}&\frac{r_{n-j+1}}{E}\leq y\leq \frac{r_{n-j}}{E}\\
&\quad\text{for }j=1, \ldots, \,n-1;
\end{split}\\
q^{n}Ey-\sum_{i=0}^{n-1}q^{i}(q-1)r_{n-i} & 
\dfrac{r_{1}}{E}\leq y.
\end{cases}\nonumber\end{equation}
\end{lemma}

We are to consider the higher ramification subgroups of $G(K_{m}/K).$
Let us assume the conditions in Lemma~\ref{l2215} from now on.

We would like to know the structure of the wild ramification subgroup of $K_{1}/K$ and how this group acts on the generators $\xi_{-1}$ and $\xi_{1}$ of $K_{1}/K.$
Since $K''/K$ and $K_{1}/K_{1}'$ are tamely ramified, the natural projection $G(K_{1}/K)\twoheadrightarrow G(K_{1}'/K)$ induces an isomorphism
$G(K_{1}/K)_{1} \cong G(K_{1}'/K)_{1} = G(K_{1}'/K'')_{1}.$
From the $\psi$-functions, we replace the indices as 
\begin{align}G(K_{1}/K)_{r_{1}} \cong G(K_{1}'/K)_{\frac{r_{1}}{e_{1}}} = G(K_{1}'/K'')_{\frac{r_{1}}{e_{1}}}.\label{f22112}\end{align}
From Proposition~\ref{p221}, we know $G(K_{1}'/K'')_{\frac{r_{1}}{e_{1}}}\cong (\Bbb{Z}/p\Bbb{Z})^{s}$ and thus $G(K_{1}/K)_{r_{1}}\cong (\Bbb{Z}/p\Bbb{Z})^{s}.$

For an element $\sigma\in G(K_{1}/K)_{r_{1}},$ note that $\sigma$ is characterized by $\sigma(\xi_{1})$ because the isomorphism (\ref{f22112}) indicates that $\sigma$ fixes $K''.$  
Then we have
\begin{align}
v_{K_{1}}(\sigma(\xi_{1})-\xi_{1})&= v_{K_{1}}(\sigma(\xi_{1})\xi_{1}^{-1}-1)+v_{K_{1}}(\xi_{1})\nonumber\\
&\geq r_{1}+v_{K_{1}}(\xi_{1})\nonumber\\
&= Eq(-\frac{v_{1}-v_{0}}{q-1})
=v_{K_{1}}(\xi_{-1})
.\nonumber\end{align}
From $\sigma(\xi_{1})-\xi_{1}\in \phi[t]$ and Proposition~\ref{p2111}, we have $\sigma(\xi_{1})=\xi_{1}+u\cdot \xi_{-1}$ for some $u\in \Bbb{F}_{q}.$ 
This defines a morphism $G(K_{1}/K)_{r_{1}}\to \Bbb{F}_{q};$ $\sigma\mapsto u,$ which is injective. 
Comparing the cardinalities of the domain and codomain, we conclude that this map is an isomorphism. 

For an integer $2\leq l\leq m,$ let us consider the wild ramification subgroup of $K_{l}/K_{l-1}$ and how this group acts on the generator $\xi_{l}$ of $K_{l}/K_{l-1}$ 
(note that $\xi_{-l}\in K_{l-1}$ by Lemma~\ref{l2211}). 
We know from Proposition~\ref{p221} that the wild ramification subgroup of $K_{l}/K_{l-1}$ is $G(K_{l}/K_{l-1})_{r_{l}}$ which is isomorphic to $(\Bbb{Z}/p\Bbb{Z})^{s}.$
For an element $\sigma\in G(K_{l}/K_{l-1})_{r_{l}},$ we similarly have 
\begin{align}
    v_{K_{l}}(\sigma(\xi_{l})-\xi_{l})&= v_{K_{l}}(\sigma(\xi_{l})\xi_{l}^{-1}-1)+v_{K_{l}}(\xi_{l})\nonumber\\
    &\geq r_{l}+v_{K_{l}}(\xi_{l})\nonumber\\
    &= \left(-Eq^{l}\cdot \frac{v_{1}-v_{0}}{q-1}-v_{K_{l}}(\xi_{l})\right)+v_{K_{l}}(\xi_{l})=v_{K_{l}}(\xi_{-1}), \nonumber
\end{align}
which shows $\sigma(\xi_{l})=\xi_{l}+u\cdot \xi_{-1}$ for some $u\in \Bbb{F}_{q}.$ This defines an isomorphism $G(K_{l}/K_{l-1})_{r_{l}}\to \Bbb{F}_{q};$ $\sigma\mapsto u.$

We are now ready to state the main theorem for infinite primes.
\begin{theorem}\label{p2211}
Assume $v_{\infty}(\bm{j})<v_{0}q$ and $p\nmid v_{\infty}(\bm{j}).$ 
Let $m$ be the integer such that $v_{\infty}(\bm{j})\in(v_{0}q^{m+1},v_{0}q^{m}).$ 
\begin{enumerate}[\rm{(}1)]

\item 
For integers $l$ and $n$ satisfying $1\leq l\leq m$ and $l\leq n,$ put $R_{n}^{l}\coloneqq \psi_{K_{n}/K}(\frac{r_{l}}{E}).$ 
We set $K_{0}\coloneqq K.$ 
Then the natural projection $G(K_{n}/K_{l-1})\to G(K_{l}/K_{l-1})$ induces an isomorphism $G(K_{n}/K_{l-1})_{R_{n}^{l}}\cong G(K_{l}/K_{l-1})_{R_{l}^{l}}\cong (\Bbb{Z}/p\Bbb{Z})^{s}.$ 
Let $\sigma_{l,u}$ for $u\in \Bbb{F}_{q}$ be the element in $G(K_{l}/K_{l-1})_{R_{l}^{l}}$ characterized by 
\[\sigma_{l,u}(\xi_{-l})=\xi_{-l}\text{ and }\sigma_{l,u}(\xi_{l})=\xi_{l}+u\cdot \xi_{-1}.\]
Denote $\sigma_{l,u}\in G(K_{n}/K_{l-1})_{R_{n}^{l}}$ again its image under the isomorphism. 
Then $\sigma_{l,u}(\xi_{-n})=\xi_{-n}$ and $\sigma_{l,u}(\xi_{n})=\xi_{n}+u\cdot \xi_{-(n-l+1)}.$  

\item
The wild ramification subgroup $G(K_{m}/K)_{1}$ of $K_{m}/K$ is isomorphic to $(\Bbb{Z}/p\Bbb{Z})^{sm}.$
\end{enumerate}
\end{theorem}

\begin{proof}
(1) We first show the results for $l=1.$ 
The case $n=1$ is known. 
Assume (1) for $n-1.$ 
If $n\geq m+1,$ then $K_{n}=K_{n-1}$ by Lemma~\ref{l2211}, so the claim follows similarly as in the case $n\leq m.$ 
Assume $n\leq m.$ 
We have $G(K_{n}/K)_{R_{n}^{1}}\cap G(K_{n}/K_{n-1})=G(K_{n}/K_{n-1})_{R_{n}^{1}}$ by \cite[Chapter~IV, Proposition~2]{Se}.
As the ramification break of $K_{n}/K_{n-1}$ is $r_{n}$ and $r_{n}<R_{n}^{1},$ we have $G(K_{n}/K)_{R_{n}^{1}}\cap G(K_{n}/K_{n-1})=1.$ 
Notice $G(K_{n}/K)^{r_{1}/E}=G(K_{n}/K)_{R_{n}^{1}}.$ 
Hence $G(K_{n}/K)^{r_{1}/E}=G(K_{n}/K)^{r_{1}/E}G(K_{n}/K_{n-1})/G(K_{n}/K_{n-1}).$
By \cite[Chapter~IV, Proposition~14]{Se}, we have an isomorphism $G(K_{n}/K)^{r_{1}/E}\cong G(K_{n-1}/K)^{r_{1}/E}.$
By the $\psi$-functions, this is the isomorphism $G(K_{n}/K)_{R_{n}^{1}}\cong G(K_{n-1}/K)_{R_{n-1}^{1}}.$ 
The first claim follows.

As for the Galois action, by induction hypothesis, we know $\phi_{t}(\sigma_{1,u}(\xi_{-n})-\xi_{-n})=0$ and thus $\sigma_{1,u}(\xi_{-n})-\xi_{-n}\in \phi[t].$ Similarly, we have $\sigma_{1,u}(\xi_{n})-\xi_{n}-u\cdot \xi_{-n}\in \phi[t].$
So $\sigma_{1,u}(\xi_{-n})-\xi_{-n}=u'\cdot \xi_{-1}+u''\cdot \xi_{1}$ for $u',$ $u''\in \Bbb{F}_{q}.$ If $u''\neq 0,$ we obtain
\begin{equation}\begin{split}R_{n}^{1}\leq v_{K_{n}}(\sigma_{1,u}(\xi_{-n})\xi_{-n}^{-1}-1)&=v_{K_{n}}(\sigma_{1,u}(\xi_{-n})-\xi_{-n})-v_{K_{n}}(\xi_{-n})\\
&=v_{K_{n}}(\xi_{1})-v_{K_{n}}(\xi_{-n})\\
&=\frac{E}{q-1}\left(v_{\infty}(\bm{j})q^{n-1}+v_{0}((n-1)q^{n+1}-nq^{n})\right)<0,\end{split}
\nonumber
\end{equation}
which is a contradiction.
Similarly, we can show $u'=0$ and thus $\sigma_{1,u}(\xi_{-n})=\xi_{-n}.$
For $\sigma_{1,u}(\xi_{n}),$ we have
\begin{equation}\begin{split}
&\,\, v_{K_{n}}(\sigma_{1,u}(\xi_{n})-\xi_{n})\\
=&\,\, v_{K_{n}}(\sigma_{1,u}(\xi_{n})\xi_{n}^{-1}-1)+v_{K_{n}}(\xi_{n})\\
\geq &\,\, R_{n}^{1}+v_{K_{n}}(\xi_{n})\\
=&\,\, \frac{E}{q-1}\left(-v_{\infty}(\bm{j})-v_{0}((n-1)q^{n+1}-nq^{n})\right)-Eq^{n}\left(\frac{v_{2}+v_{1}(q^{n}-q-1)}{(q-1)q^{n}}\right)\\
=&\,\, Eq^{n}\left(-v_{0}(n-1)-\frac{v_{1}-v_{0}}{q-1}\right)=v_{K_{n}}(\xi_{-n}).\end{split}\label{f22113}
\end{equation}
Since $v_{\infty}(\xi_{-n})>v_{\infty}(\xi_{-1})>v_{\infty}(\xi_{1})$ by Proposition~\ref{p2111}, we have $\sigma_{1,u}(\xi_{n})=\xi_{n}+u\cdot \xi_{-n}.$

Then we show the case for all $l.$ 
We again use induction.
Similar to the proof in the case $l=1,$ we have the isomorphism $G(K_{n}/K_{l-1})_{R_{n}^{l}}\cong G(K_{n-1}/K_{l-1})_{R_{n-1}^{l}}$ by $R_{n}^{l}>r_{n}.$
We can show $\sigma_{l,u}(\xi_{-n})-\xi_{-n},$ $\sigma_{l,u}(\xi_{n})-\xi_{n}-u\cdot \xi_{-(n-l+1)}\in \phi[t].$
Similar calculations in the case $l=1$ show that they vanish.

(2) From Lemma~\ref{l2215}, the wild ramification subgroup $G(K_{m}/K)_{1}$ is equal to $G(K_{m}/K)_{r_{m}}.$
By (1), it is generated by $\{\sigma_{l,u}\,|\,1\leq l\leq m,\,\,u\in \Bbb{F}_{q}\}.$ 
For a basis $\{\xi_{-m}, \ldots, \xi_{-1}, \xi_{m}, \ldots, \xi_{1}\}$ of $\phi[t^{m}]$ with the order according to the valuations, we can identify each $\sigma_{l,u}$ as the representation matrix 
\[\begin{pmatrix}I_{m}&0\\
u\cdot A_{m,l}&I_{m}
\end{pmatrix}
\]
with respect to this basis. 
Here $I_{m}$ denotes the $m\times m$ identity matrix and $A_{m,l}$ is the $m\times m$ matrix defined by $(\delta_{i,j-l+1})_{ij}$ with the Kronecker delta $\delta.$ 
This gives a monomorphism $G(K_{m}/K)_{r_{m}}\to \text{GL}_{2m}(\Bbb{F}_{q}).$ Clearly its image is isomorphic to $(\Bbb{Z}/p\Bbb{Z})^{sm}.$
\end{proof}

\subsection{Finite primes}\label{s222}
This section deals with the finite prime case and is carried out in parallel with the previous section.
Let $K$ be the completion of $F$ at a finite prime $v_{f}$ not dividing $t.$
Assume $v_{f}(\bm{j})<0.$
Due to Proposition~\ref{p2121}~(1), we know $K_{n+1}=K_{n}(\xi_{n+1},\xi_{-n-1}).$
Put $K_{n,-}\coloneqq K_{n}(\xi_{-n-1})$ (resp.\ $K_{n,+}\coloneqq K_{n}(\xi_{n+1})$) and it is the splitting field of $\phi_{t}(X)-\xi_{-n}$ (resp.\ $\phi_{t}(X)-\xi_{n}$) $\in K_{n}[X]$.
\begin{lemma}\label{l2221}
Assume $v_{f}(\bm{j})<0.$
\begin{enumerate}[\rm{(}1)]
\item Let $K'$ denote the splitting field of $\eta(X)\coloneqq a_{1}X^{q}+tX.$
Then $K'=K(\xi_{-1}).$
\item For all $n\geq 1$ and any root $x_{-}'$ of the polynomial $\eta(X)-\xi_{-n}$, we have $K_{n}(x_{-}')= K_{n,-}$.
For all $n\geq 1$ and any root $x_{+}'$ of the polynomial $\eta(X)-\xi_{n}$, we have $K_{n}(x_{+}')= K_{n,+}$, unless $q=2$ and $n=1$.
\item
Let $K_{1}'$ be the splitting field of the degree $q$ polynomial $\widetilde{\Eta}(X)=X^{q}-\beta X-\beta\in K'[X]$ with $\beta=\frac{\xi_{-1}^{q^{2}-1}a_{2}}{t}.$ Then
\begin{itemize}
\item
$K_{1}'/K'$ is a subextension of $K_{1}/K';$
\item $K_{1}/K_{1}'$ is a compositum of Kummer extensions.
\end{itemize}
\end{enumerate}
\end{lemma}
\begin{remark}\label{r2221}
Let us warn that the extension $\varinjlim_{n} K_{n}/K$ is not finite.
Hence there is no analogue of Lemma~\ref{l2211}.
\end{remark}

\begin{proof}
Because the proof of this lemma is essentially the same as those of Lemmas~\ref{l2212}, \ref{l2213}, and \ref{l2214}, we only give an outline.
For (1), let $x_{i}'$ for $i=1,\ldots,q-1$ denote the nonzero roots of $\eta(X).$
Then $|x_{i}'-x_{j}'| = q^{\frac{v_{1}}{q-1}}$ for $i \neq j.$
Since 
\[\eta(\xi_{-1}) = a_{1}\xi_{-1}\prod_{i=1}^{q-1}(\xi_{-1} - x_{i}') = -a_{2}\xi_{-1}^{q^{2}},\] 
we have
\[\left|\prod_{i=1}^{q-1} (\xi_{-1} - x_{i}')\right| = \left|a_{1}^{-1}a_{2}\cdot \xi_{-1}^{q^{2}-1}\right|=q^{v_{f}(\bm{j})+v_{1}}.\]
As $v_{f}(\bm{j})<0,$ there exists some $x_{i}'$ such that \[|\xi_{-1} - x_{i}'| \leq q^{\frac{v_{f}(\bm{j})+v_{1}}{q-1}} < q^{\frac{v_{1}}{q-1}} = |x_{i}' - x_{j}'|.\]
By Krasner's lemma, we have $K' \subset K(\xi_{-1}).$
Note that the conjugates of $\xi_{-1}$ are $u \cdot \xi_{-1}$ for $u \in \Bbb{F}_{q}^{\times}.$
The inverse inclusion follows from $|\xi_{-1} - x_{i}'| < |\xi_{-1} - u \cdot \xi_{-1}|$ for $u \in \Bbb{F}_{q} \setminus \{0,1\}$ and Krasner's lemma.

We use the following polynomials in $K_{n}[X]$ as in (\ref{f2217}) to show (2) and (3): 
\begin{alignat}{2}
\Phi(X)&=b_{2}X^{q^{2}}+X^{q}+b_{0}X&&\coloneqq b'(\phi_{t}(X/b)),\nonumber\\
\Phi_{n}(X)&=b_{2}X^{q^{2}}+X^{q}+b_{0}X+c_{n}&&\coloneqq b'(\phi_{t}(X/b)-\xi_{n}),\nonumber\\
\Phi_{-n}(X)&=b_{2}X^{q^{2}}+X^{q}+b_{0}X+c_{-n}&&\coloneqq b'(\phi_{t}(X/b)-\xi_{-n}),\nonumber\\
\Eta(X)&=X^{q}+b_{0}X&&\coloneqq b'(\eta(X/b)),\nonumber\\
\Eta_{n}(X)&=X^{q}+b_{0}X+c_{n}&&\coloneqq b'(\eta(X/b)-\xi_{n}),\nonumber\\
\Eta_{-n}(X)&=X^{q}+b_{0}X+c_{-n}&&\coloneqq b'(\eta(X/b)-\xi_{-n}).\nonumber
\end{alignat}
Here we choose $b\in K^{\prime \times}$ satisfying $v_{f}(b)=\frac{v_{1}}{q-1}$ and put $b'=b^{q}/a_{1}.$
The splitting fields of 
$\Phi_{\pm n}(X)$ and $\Eta_{\pm n}(X)$ are respectively $K_{n,\pm}$ and $K_{n}
(x_{\pm}').$
For $\Phi(X),$ we have $v_{f}(b_{0}),$ $v_{f}(1),$ $v_{f}(b_{2})$ being respectively
\[0,\,\,0,\,\,-v_{f}(\bm{j}).\]
For $\Phi_{n}(X),$ we have $v_{f}(c_{n}) = \frac{v_{f}(\bm{j})}{(q-1)q^{n}}.$
For $\Phi_{-n}(X),$ we have $v_{f}(c_{-n})=0.$

We show (2). 
Put $x = b \xi_{-(n+1)}.$
Let $x_{i}'$ for $i=1,\ldots,q$ denote the roots of $\Eta_{-n}(X).$
We have $|x_{i}'-x_{j}'| = 1$ for $i \neq j$ as $x_{i}' - x_{j}'$ is a nonzero root of $\Eta(X).$
Since 
\[\Eta_{-n}(x) = \prod_{i=1}^{q}(x-x_{i}') = - b_{2}x^{q^{2}},\] we have 
\[\left|\prod_{i=1}^{q}(x-x_{i}')\right| = \left|b_{2}x^{q^{2}}\right| = q^{v_{f}(\bm{j})}.\]
As $v_{f}(\bm{j}) < 0,$ there exists some $x_{i}'$ such that $|x - x_{i}'| < 1 = |x_{i}' - x_{j}'|$ for all $j \neq i.$
Krasner's lemma implies that the splitting fields of $\Phi_{-n}(X)$ and $\Eta_{-n}(X)$ are the same.

Next, put $x = b\xi_{n+1}.$
Let $x_{i}'$ for $i=1,\ldots,q$ denote the roots of $\Eta_{n}(X).$
We have $|x_{i}'-x_{j}'| = 1$ for $i\neq j.$
By considering $\Eta_{n}(x),$ we obtain again 
\[\left|\prod_{i=1}^{q}(x-x_{i}')\right| = \left|b_{2}x^{q^{2}}\right|.\]
Note that we have \[v_{f}(b_{2}x^{q^{2}}) = - v_{f}(\bm{j})\left(1-\frac{1}{(q-1)q^{n-1}}\right) > 0\]
unless $q = 2$ and $n = 1.$
Hence 
\[\left|\prod_{i=1}^{q}(x-x_{i}')\right|<1.\]
and some $x_{i}'$ satisfies $|x-x_{i}'| < 1 = |x_{i}' - x_{j}'|$ for $j \neq i.$
Krasner's lemma implies that the splitting fields of  $\Phi_{n}(X)$ and $\Eta_{n}(X)$ are the same.

The proof of (3) is exactly the same as that of Lemma~\ref{l2214}.
\end{proof}

\begin{remark}
This is a remark on $q=2,$ $n=1$ case of the proof.
In this case, our method of proving $K_{n}(x')= K_{n,+}$ in Lemma~\ref{l2221}~(2) breaks down since $\left|\prod_{i=1}^{q}(x-x_{i}')\right| = 1.$
Moreover, there is no root $x_i'$ of $\Eta_{1}(X)$ such that $|x-x_i'|<1.$
Indeed, we have $|x-x_1'||x-x_2'| = 1$ from $v_{f}(b_{2}x^{q^{2}})=0.$
Suppose that there is $i$ such that $|x-x_i'|<1.$
We may assume $i=1.$
Then we obtain a contradictory inequality
\[ 1<|x-x_2'|=|(x-x_1')-(x_1'-x_2')| \leq \max\{|x-x_1'|, |x_1'-x_2'|\}=1. \]
\end{remark}

Assume $p \nmid v_{f}(\bm{j})$ from now on.
We work out the $\psi$-functions of $K_{1}/K$ and $K_{n+1}/K_{n}.$
Note that $q$ divides the ramification index of $K_{1}/K$ since $v_{f}(\xi_{1})=\frac{v_{f}(\bm{j})-v_{1}q}{(q-1)q}.$
Let $E$ be the integer such that $Eq$ is the ramification index of $K_{1}/K.$

\begin{lemma}\label{l2223}
Assume $v_{f}(\bm{j}) < 0,$ $p \nmid v_{f}(\bm{j}).$
Put $r \coloneqq E\frac{-v_{f}(\bm{j})}{q-1}.$
\begin{enumerate}[\rm{(}1)]
    \item 
Put $e_{0}\coloneqq e_{K'/K}$ and $e_{1}\coloneqq e_{K_{1}/K_{1}'}$ respectively to be the ramification indices of $K'/K$ and $K_{1}/K_{1}'.$
Let $e'$ denote the ramification index of the extension $K''$ of $K'$ generated by all roots of $X^{q} - \beta X.$
Then we have $E=e_{1}e'e_{0}$ and 
\[\psi_{K_{1}/K}(y) = \begin{cases}y & -1 \leq y \leq 0;\\
Ey & 0 \leq y \leq \dfrac{r}{E};\\
qEy - (q-1)r & \dfrac{r}{E} \leq y.
\end{cases}\]
    \item Assume $q \neq 2.$
\begin{enumerate}[\rm{(}i)]
    \item For $n \geq 1,$ the ramification index of $K_{n}/K$ is $Eq^{n};$
    \item For $n\geq 1,$ we have 
\[\psi_{K_{n+1}/K_{n}}(y) = \begin{cases} y & -1 \leq y \leq r;\\
qy - (q-1)r & r \leq y.\end{cases}\]
\end{enumerate}

\end{enumerate}
\end{lemma}
\begin{proof}
As the extensions $K'/K$ and $K''/K'$ are Kummer and the extension $K_{1}/K_{1}'$ is a compositum of Kummer extensions, we have
\[\psi_{K''/K}(y) = \begin{cases}
y & -1 \leq y \leq 0;\\
e_{0}e'y & 0 \leq y,
\end{cases}\text{ and }\psi_{K_{1}/K_{1}'}(y) = \begin{cases}
y & -1 \leq y \leq 0;\\
e_{1}y & 0 \leq y.
\end{cases}\]

Consider the polynomial $\widehat{\Eta}(X)=\omega^{-q}\cdot \widetilde{\Eta}(\omega \cdot X)=X^{q}-X-\omega^{-1}\in K''[X],$ where $\widetilde{\Eta}(X)$ is the polynomial defined in Lemma~\ref{l2221}~(3) and $\omega$ is a nonzero root of $X^{q}-\beta X$ with valuation $\frac{-v_{f}(\bm{j})}{q-1}.$
Then $K_{1}'/K''$ is generated by one root of $\widehat{\Eta}(X)\in K''[X].$ 
Since $p\nmid -v_{K''}(\omega^{-1})$ and  $-v_{K''}(\omega^{-1})>0$ with  $v_{K''}=e'e_{0}v_{f},$ we can apply Proposition~\ref{p221} to $\widehat{\Eta}(X).$ 
Thus we know $K_{1}'/K''$ is a degree $q$ totally ramified Galois extension with ramification break $v_{K''}(\omega).$
So $E=e_{1}e'e_{0}.$
Note $v_{K''}(\omega) = Er/e_{1}.$ 
We have the $\psi$-function 
\begin{align}\psi_{K_{1}'/K''}(y)=\begin{cases}y & -1\leq y\leq \dfrac{Er}{e_{1}};\\
qy-(q-1)\dfrac{Er}{e_{1}} & \dfrac{Er}{e_{1}}\leq y.
\end{cases}
\nonumber\end{align}
By Lemma~\ref{l221}, (1) follows.

We show (i) of (2).
The case $n=1$ is known.
Assume that the claim holds for $n.$
We show that $K_{n+1}/K$ has ramification index $Eq^{n+1}.$ 
It suffices to show that the ramification index of $K_{n+1}/K_{n}$ is $q.$ 
By Lemma~\ref{l2221}~(2), $K_{n+1}/K_{n,+}$ is the extension of $K_{n,+}$ generated by a root of $\Eta_{-n}(X).$ 
We can apply  \cite[Chapter~II,\,Proposition~3.2]{FV} to a root of $\Eta_{-n}(X).$ 
Thus $K_{n+1}/K_{n,+}$ is an unramified Galois extension.
For $K_{n,+}/K_{n},$ take $1/b$ to be a root of $\eta(X)$ so that $b_{0}=-1$ in $\Eta_{n}(X).$ 
Since $p\nmid -v_{K_{n}}(c_{n})$ and  $-v_{K_{n}}(c_{n})=E\frac{-v_{f}(\bm{j})}{q-1}>0$ with $v_{K_{n}}=Eq^{n}v_{f},$ we are able to apply Proposition~\ref{p221} to $\Eta_{n}(X).$ 
We know that $K_{n,+}/K_{n}$ is a degree $q$ totally ramified Galois extension. 
Hence (i) follows.

We also obtain that the ramification break of $K_{n,+}/K_{n}$ (thus of $K_{n+1}/K_{n}$) is $r = -v_{K_{n}}(c_{n})=E\frac{-v_{f}(\bm{j})}{q-1}.$ 
We have the $\psi$-function 
\[\psi_{K_{n+1}/K_{n}}(y)=\psi_{K_{n,+}/K_{n}}(y)=\begin{cases}
y & -1\leq y\leq r;\\
qy-(q-1)r & r\leq y,
\end{cases}\]
and (ii) of (2) follows. 
\end{proof}

Combining with Lemma~\ref{l221}, we can show by induction that
\begin{lemma}\label{l2222}
Assume $q \neq 2,$ $v_{f}\nmid t,$ $v_{f}(\bm{j})<0,$ and $p\nmid v_{f}(\bm{j}).$
For any $n\geq 1,$ we have
\begin{align}\psi_{K_{n}/K}(y)=\begin{cases}
y & -1\leq y\leq 0;\\
Ey  & 0\leq y\leq \dfrac{r}{E};\\
Eq^{n}y-(q^{n}-1)r & \dfrac{r}{E}\leq y. 
\end{cases}\nonumber\end{align}
\end{lemma}

We are to consider the higher ramification subgroups of $G(K_{n}/K).$
Let us assume the conditions in Lemma~\ref{l2222} from now on.

Similarly to the infinite prime case, we have isomorphisms \[G(K_{1}/K)_{r}\cong G(K_{1}'/K)_{\frac{r}{e_{1}}}=G(K_{1}'/K'')_{\frac{r}{e_{1}}}\cong (\Bbb{Z}/p\Bbb{Z})^{s}.\]
The wild ramification subgroup $G(K_{1}/K)_{1}$ is equal to $G(K_{1}/K)_{r}.$ 
For any element $\sigma\in G(K_{1}/K)_{r},$ there exists $u\in \Bbb{F}_{q}$ such that 
\[\sigma(\xi_{-1})=\xi_{-1}\text{ and }\sigma(\xi_{1})=\xi_{1}+u\cdot \xi_{-1}.\]
This defines an isomorphism $G(K_{1}/K)_{r}\to \Bbb{F}_{q}\cong (\Bbb{Z}/p\Bbb{Z})^{s}.$
For $n\geq 2,$ we have $G(K_{n}/K_{n-1})_{1}=G(K_{n}/K_{n-1})_{r}\cong (\Bbb{Z}/p\Bbb{Z})^{s}.$ Indeed, any $\sigma\in G(K_{n}/K_{n-1})_{1}$ is characterized by 
\begin{align}\sigma(\xi_{-n})=\xi_{-n}\text{ and }\sigma(\xi_{n})=\xi_{n}+u\cdot \xi_{-1}
\nonumber\end{align}
for some $u\in\Bbb{F}_{q}.$ The map $\sigma\mapsto u$ gives an isomorphism.

We are now ready to state the main theorem for finite primes.
\begin{theorem}\label{p2221}
Assume $q \neq 2,$ $v_{f}\nmid t,$ $v_{f}(\bm{j})<0,$ and $p\nmid v_{f}(\bm{j}).$
\begin{enumerate}[\rm{(}1)]
\item 
The wild ramification subgroup $G(K_{n}/K)_{1}=G(K_{n}/K)_{r}$ of $K_{n}/K$ is generated by $\sigma_{l,u}$ for $1\leq l \leq n$ and $u\in \Bbb{F}_{q},$ where $\sigma_{l,u}$ is characterized by  
\begin{align}\sigma_{l,u}(\xi_{-i})=\xi_{-i}\text{ for }1\leq i\leq n\text{ and }\sigma_{l,u}(\xi_{i})=\begin{cases}\xi_{i}&1\leq i\leq l-1;\\
\xi_{i}+u\cdot \xi_{-(i-l+1)}&i\geq l.
\end{cases}\nonumber\end{align}

\item 
$G(K_{n}/K)_{1}$ is isomorphic to $(\Bbb{Z}/p\Bbb{Z})^{sn}.$
\end{enumerate}
\end{theorem}

\begin{proof}
(1) The case $n=1$ is known. 
Assume (1) for the case $n-1.$ 
Let $\sigma$ be an element in $G(K_{n}/K)_{r}.$ 
By the induction hypothesis, we know $\sigma(\xi_{-n})-\xi_{-n}\in \phi[t].$ 
Similar to the proof of Theorem~\ref{p2211}, if $\sigma(\xi_{-n})-\xi_{-n}\neq 0,$ there is a contradiction $r\leq v_{K_{n}}(\sigma(\xi_{-n})\xi_{-n}^{-1}-1)<r.$
Thus $\sigma$ fixes $\xi_{-1}, \ldots, \xi_{-n}.$

Since $\phi_{t^{i}}(\xi_{\pm n})=\xi_{\pm (n-i)}$ for $1\leq i\leq n-1,$ we know $K_{n}=K(\xi_{-n},\xi_{n}).$ 
Hence $\sigma\in G(K_{n}/K)_{r}$ is characterized by $\sigma(\xi_{n}).$
For $\sigma(\xi_{n}),$ similar to (\ref{f22113}), we have by Proposition~\ref{p2121}~(1),
\begin{align}
v_{K_{n}}(\sigma(\xi_{n})-\xi_{n})&= v_{K_{n}}(\sigma(\xi_{n})\xi_{n}^{-1}-1)+v_{K_{n}}(\xi_{n})\nonumber\\
&\geq r+v_{K_{n}}(\xi_{n})=v_{K_{n}}(\xi_{-1})=\cdots =v_{K_{n}}(\xi_{-n}).\nonumber
\end{align}
Thus $\sigma(\xi_{n})-\xi_{n}$ is an $\Bbb{F}_{q}$-linear combination of $\xi_{-1}, \ldots, \xi_{-n}.$ 
If we put $\sigma(\xi_{n})-\xi_{n}=\sum_{i=1}^{n}u_{-i}\cdot \xi_{-i}$ for $u_{-i}\in \Bbb{F}_{q},$ then $\sigma=\sigma_{1,u_{-1}}\cdots \sigma_{n,u_{-n}}.$ 
This shows (1) for $n.$

(2) can be shown in the same way as the proof of Theorem~\ref{p2211}~(2). 
\end{proof}

\subsection{Tamely ramified extensions by division points}\label{s23}
This subsection is concerned with the case $v_{\infty}(\bm{j})\in [v_{0}q,+\infty)$ for an infinite prime $v_{\infty}$ and the case $v_{f}(\bm{j})\in [0,+\infty)$ for a finite prime $v_{f}$ not dividing $t.$
We are to show that the extensions of local fields given by $t^{n}$-division points are at worst tamely ramified.
Let us follow the notation in Sections~\ref{s21}, \ref{s221}, and \ref{s222}.  
\begin{lemma}\label{l231}
Let $v_{\infty}$ denote an infinite prime of $F$ and its corresponding valuation such that $v_{\infty}(\bm{j})\in [v_{0}q,+\infty).$
Let $K$ denote the completion of $F$ at 
$v_{\infty}.$ 
\begin{enumerate}[\rm{(}1)]
\item
The extension $K_{1}/K$ is at worst tamely ramified. The ramification index of $K_{1}/K$ divides $q^{2}-1;$
\item $K_{n}=K_{1}.$
\end{enumerate}
\end{lemma}

Under the assumption of this lemma, we have $K_{n}/K=K_{1}/K$ which is at worst tamely ramified.

\begin{proof}
For (1), the Newton polygon of $\phi_{t}(X)$ has one segment with slope $\frac{v_{2}-v_{0}}{q^{2}-1}.$ Let $M$ be an extension of $K$ with ramification index $e_{M/K}$ being $q^{2}-1.$ We can take $b\in M$ such that $v_{\infty}(b)=\frac{v_{2}-v_{0}}{q^{2}-1}.$ 
With $b'=b^{q^{2}}/a_{2},$ modify $\phi_{t}(X)$ to be
\[\Phi'(X)=X^{q^{2}}+b_{1}X^{q}+b_{0}X\coloneqq b'\phi_{t}(X/b).\]
The valuations $v_{\infty}(b_{0}),$ $v_{\infty}(b_{1}),$ $v_{\infty}(1)$ are respectively
\[0,\,\,\frac{v_{\infty}(\bm{j})-v_{0}q}{q+1}\geq 0,\,\,0.\]
Thus $\Phi'(X)$ is a monic polynomial whose reduction is separable.
We know the extension $M_{1}$ of $M$ generated by the roots of $\Phi'(X)$ is unramified by \cite[Chapter~II,\,Proposition~3.2]{FV}.
Since $K_{1}/K$ is a subextension of $M_{1}/K,$ the ramification index $e_{K_{1}/K}$ divides the ramification index $e_{M/K}=e_{M_{1}/K}=q^{2}-1.$  
So $K_{1}/K$ is tamely ramified.

As for (2), it suffices to show that $\xi_{-n-1}$ and $\xi_{n+1}$ belong to $K_{n}.$ 
This follows from Proposition~\ref{p2112}~(2) in the same way as Lemma~\ref{l2211} follows from Lemma~\ref{l2111}.
\end{proof}

Consider the finite prime case then.

\begin{lemma}\label{l232}
Let $v_{f}$ denote a finite prime of $F$ not dividing $t$ and its corresponding valuation such that $v_{f}(\bm{j})\in [0,+\infty).$
Let $K$ denote the completion of $F$ at 
$v_{f}.$ 
\begin{enumerate}[\rm{(}1)]
\item
The extension $K_{1}/K$ is at worst tamely ramified with the ramification index dividing $q^{2}-1;$
\item
Let $K_{n,-}$ \text{\rm{(}}resp.\ $K_{n,+}$\text{\rm{)}} denote the extension of $K_{n}$ generated by $\xi_{-n-1}$ \text{\rm{(}}resp.\ $\xi_{n+1}$\text{\rm{)}}.
Then the extensions $K_{n,-}/K_{n}$ and $K_{n,+}/K_{n}$ are unramified.
\end{enumerate}
\end{lemma}

Under the assumption of this lemma, 
we see that the extension $K_{n}/K$ is a tamely ramified extension $K_{1}/K$ followed by an unramified extension $K_{n}/K_{1}.$

\begin{proof}
The proof of (1) is similar to that of Lemma~\ref{l231}~(1).

For the proof of (2),  we use Proposition~\ref{p2121}~(2) and follow the same method as in the proof of Lemma~\ref{l231}~(1). 
We can take $b\in K_{1}$ such that $v_{f}(b)=\frac{v_{2}}{q^{2}-1}.$ 
With $b'\coloneqq b^{q^{2}}/a_{2},$ modify $\phi_{t}(X)-\xi_{-n}$ and $\phi_{t}(X)-\xi_{n}$ respectively to be  $b'(\phi_{t}(X/b)-\xi_{-n})$ and $b'(\phi_{t}(X/b)-\xi_{n})$ so that the valuations of the coefficients of $1,$ $X,$ $X^{q},$ $X^{q^{2}}$ are respectively 
\[0,\,\,0,\,\,\frac{v_{f}(\bm{j})}{q+1},\,\,0.\]
Then (2) follows from \cite[Chapter~II,\,Proposition~3.2]{FV}.
\end{proof}

Let $v$ be an infinite prime $v_{\infty}$ satisfying the assumption of Lemma~\ref{l231} or a finite prime $v_{f}$ not lying above $t$ satisfying the assumption of Lemma~\ref{l232}.
Put $v_{0}=v(t).$
We finish this section by determining the $\psi$-function of $K_{1}/K$ in the case $v(\bm{j}) \in (v_{0}q,+\infty).$
Since $K_1 / K$ is tamely ramified, determining the $\psi$-function is equivalent to determining the ramification index.

If we follow the discussion in Sections~\ref{s221} and \ref{s222}, it seems natural to ask whether there is an analogue of Lemmas~\ref{l2212} and \ref{l2221} (1) for such $v$.
Namely, does $K_1$ contain the splitting field of some binomial whose terms come from $\phi_t(X)$?
This is answered affirmatively in the following lemma under the assumption that $K$ contains all $(q^{2} - 1)$-st roots of unity.

\begin{lemma}\label{l233}
Let $L$ be the extension of $K$ generated by the roots of $t X + a_2 X^{q^2}$.
Assume $v(\boldsymbol{j}) \in (v_0 q, + \infty)$ and that $K$ contains $\mathbb{F}_{q^2}$ so that the extension $L / K$ is Kummer.
Then we have $L = K_1$.
\end{lemma}

\begin{proof}
The proof is carried out by the strategy used in that of Lemma~\ref{l2212}.
We give an outline.
Let $x_{i}'$ for $i=1,\ldots,q^{2}-1$ denote the nonzero roots of $tX + a_{2}X^{q^{2}}.$ 
Then $|x_{i}' - x_{j}'| = q^{\frac{v_{2} - v_{0}}{q^{2}-1}}$ for $i \neq j.$
Let $\xi$ denote $\xi_{-1}$ or $\xi_{1}.$
Since 
\[a_{2}\xi\prod_{i=1}^{q^{2}-1}(\xi - x_{i}') = a_{2}\xi^{q^{2}} + t\xi = - a_{1}\xi^{q},\]
we have 
\[\left|\prod_{i=1}^{q^{2}-1}(\xi - x_{i}')\right| = \left|a_{2}^{-1}a_{1}\xi^{q-1}\right| = q^{\frac{-v(\bm{j})+v_{0}q}{q+1} + (v_{2} - v_{0})}.\]
As $v(\bm{j})>v_{0}q,$ there exist some indices $i_{-}$ and $i_{+}$ satisfying 
\begin{align}|\xi_{-1} - x_{i_{-}}'| & \leq q^{\frac{-v(\bm{j})+v_{0}q}{(q+1)(q^{2}-1)}+\frac{v_{2}-v_{0}}{q^{2}-1}} < q^{\frac{v_{2}-v_{0}}{q^{2}-1}} = |x_{i_{-}}'-x_{j}'| \text{ for }j \neq i_{-},\nonumber\\
|\xi_{1} - x_{i_{+}}'| & \leq q^{\frac{-v(\bm{j})+v_{0}q}{(q+1)(q^{2}-1)}+\frac{v_{2}-v_{0}}{q^{2}-1}} < q^{\frac{v_{2}-v_{0}}{q^{2}-1}} = |x_{i_{+}}'-x_{j}'| \text{ for }j \neq i_{+}.\nonumber
\end{align}
By Kranser's lemma, we have $L \subset K_{1}.$
Note that the conjugates of $\xi_{-1}$ and $\xi_{1}$ are of the form $u \cdot \xi_{-1} + u' \cdot \xi_{1}$ for $(u, u') \in \Bbb{F}_{q}^{2} \setminus \{(0,0)\}.$
As $|\xi_{-1} - x_{i_{-}}'|<|\xi_{-1} - (u \cdot \xi_{-1} + u' \cdot \xi_{1})|$ for $(u, u') \in \Bbb{F}_{q}^{2} \setminus \{(0,0),(1,0)\},$ and $|\xi_{1} - x_{i_{+}}'|<|\xi_{1} - (u \cdot \xi_{-1} + u' \cdot \xi_{1})|$ for $(u, u') \in \Bbb{F}_{q}^{2} \setminus \{(0,0),(0,1)\},$ Kranser's lemma implies that $K(\xi_{-1}) \subset L$ and $K(\xi_{1}) \subset L.$
Hence $K_{1} \subset L.$
\end{proof}

\begin{proposition}\label{p231}
Assume $v(\bm{j}) \in (v_{0}q,+\infty).$
The ramification index of $K_{1}/K$ is $\frac{q^{2}-1}{n}$ with $n = \gcd(v(t / a_2), q^{2}-1).$
\end{proposition}

\begin{proof}
Assume for the moment that $K$ contains $\mathbb{F}_{q^2}$ so that the extension $L / K = K_{1} / K$ in Lemma~\ref{l233} is Kummer.
By replacing $\phi$ with some isomorphic Drinfeld $\Bbb{F}_{q}[t]$-module $\phi'$ over $K$ with the leading coefficient of $\phi'_{t}(X)$ having sufficiently negative valuation, we may assume that $v(t / a_2) > 0$.
Note that $n$ is unchanged under replacing $\phi$ with $\phi'$.
There exists some $\alpha' \in K$ with $v(\alpha') = n$ such that $L = K(\sqrt[q^{2}-1]{\alpha'})$ and the subextension $K(\sqrt[n]{\alpha'})/K$ of $L/K$ is unramified~(see \cite[Section~2, Lemma~6]{Bir}).
Put $\alpha = \sqrt[n]{\alpha'}.$
As $v(\alpha) = 1,$ the extension $K(\sqrt[\frac{q^{2}-1}{n}]{\alpha})/K(\alpha)$ is totally ramified with degree $\frac{q^{2}-1}{n}.$
Therefore, the ramification index of $K_{1}/K$ equals $\frac{q^{2}-1}{n}.$

For general $K$, considering the compositum $K \Bbb{F}_{q^2}$ and the compositum $K_1 \Bbb{F}_{q^2}$ instead of $K$ and $K_{1}$ respectively, and using the fact that any residue field extension is unramified, we obtain the same result on the ramification index. 
\end{proof}

\section{Functional $1$-Szpiro conjecture}\label{s31}
In this section, let $F$ be a finite extension of $\Bbb{F}_{q}(t).$ 
To obtain a function field analogue of Szpiro's conjecture, we first define at a prime $v$ of $F$ the $1$\text{\it{-conductor}} of a certain rank $2$ Drinfeld $\Bbb{F}_{q}[t]$-module $\phi$ over $F$ in Section~\ref{s311}. 
This is defined by using the $\pi$-adic Tate module for any \text{\it{degree}} $1$ polynomial $\pi$ in $\Bbb{F}_{q}[t]$ prime to $v.$ 
Using the results in Section~\ref{s22}, we will see for each prime of $F,$ the 1-conductor is independent of the choice of $\pi.$ 
Also the $1$-conductor of $\phi$ at each prime of $F$ will be explicitly calculated. 

In Section~\ref{s312}, we claim a numerical relation between the $J$-heights and the $1$-conductors of certain rank $2$ Drinfeld $\Bbb{F}_{q}[t]$-modules, where the $J$-height is considered as a replacement of the Faltings-Parshin height in Conjecture~\ref{c11}.

\subsection{1-conductors for rank 2 Drinfeld modules}\label{s311}
Let $\phi$ be a rank 2 Drinfeld $\Bbb{F}_{q}[t]$-module over $F.$ 
We denote by $\pi$ an irreducible polynomial in $\Bbb{F}_{q}[t]$ which is the uniformizer of the prime (also denoted by $\pi$) of $\Bbb{F}_{q}[t].$
If $\text{deg}\,\pi=d,$ then the polynomial $\phi_{\pi}(X)$ is of the form \[\phi_{\pi}(X)=\pi X+ a_{1}^{\pi}X^{q}+\cdots+a_{2d}^{\pi}X^{q^{2d}}\]
with $a_{i}^{\pi}\in F$ for $i=1, \ldots, 2d.$ 
Then $\phi[\pi^{n}]$ is a $2n$-dimensional $\Bbb{F}_{q}[t]/(\pi)$-vector space. 
By the $\pi$-adic Tate module $T_{\pi}$ of the Drinfeld module $\phi,$ we mean the rank $2$ free $\Bbb{F}_{q}[t]_{\pi}$-module $\varprojlim_{n}\phi[\pi^{n}].$
Denote by $\phi[\pi^{\infty}]$ the union of all $\phi[\pi^{n}]$ and by $K_{\infty}$ the extension of $K$ generated by elements in $\phi[\pi^{\infty}],$ i.e., $K_{\infty}=\varinjlim_{n} K_{n}.$ 
We keep the notation in Sections~\ref{s21} and \ref{s22}. 
We define the $1$-conductor of $\phi$ in the following manner.

\begin{lemdef}\label{d311}
Let $\phi$ be a rank $2$ Drinfeld $\Bbb{F}_{q}[t]$-module over $F.$ 
Let $v$ be a prime of $F.$
Let $\pi$ be a prime of $\Bbb{F}_{q}[t]$ and let it also denote its uniformizer. 
Assume that the polynomial $\pi$ is of degree $1$ and is not divisible by $v.$
Let $K$ be the completion of $F$ at $v.$ Put $v_{0}\coloneqq v(\pi).$
Assume 
\begin{equation}
    \begin{cases}
    \begin{split}
    &\text{either }\big(p\nmid v(\bm{j})\text{ and }v(\bm{j})<v_{0}q\big),&\\
    &\quad\quad\quad\quad\text{or }v(\bm{j})\geq v_{0}q\text{ if }v\text{ is infinite};
    \end{split}\\
    \begin{split}&\text{either }\big(q \neq 2,\,\,p\nmid v(\bm{j}),\text{ and }v(\bm{j})<0\big),&\\
    &\quad\quad\quad\quad\text{or }v(\bm{j})\geq 0\text{ if }v\text{ is finite}.\end{split}
    \end{cases}\label{f411}
\end{equation}
Let $G_{v}$ denote the absolute decomposition group  $G(K^{\text{\rm{sep}}}/K)$ of $K.$ 
Put
\begin{align}\mathfrak{f}_{v}(\phi)\coloneqq \int_{0}^{+\infty}(2-\text{\rm{rank}}_{\,\Bbb{F}_{q}[t]_{\pi}} \,T_{\pi}^{G_{v}^{y}})dy,\label{f3112}\end{align}
where $G_{v}^{y}$ denotes the $y$-th upper ramification subgroup of $G_{v}.$ 
Then 
\begin{enumerate}[\rm{(}1)]
\item
the integral $\mathfrak{f}_{v}(\phi)$ is convergent\text{\rm{;}}
\item
the value $\mathfrak{f}_{v}(\phi)$ is independent of the choice of a degree $1$ polynomial $\pi.$ 
\end{enumerate}
Define the $1$\text{\rm{-conductor}} of $\phi$ at $v$ to be the integral $\mathfrak{f}_{v}(\phi).$
\end{lemdef}

\begin{proof}
We first prove (2).  
It suffices to consider the cases where $\pi=t$ and $t-u$ for $u\in \Bbb{F}_{q}.$
Let $v$ be an infinite or finite prime of $F$ not dividing $t$ and $t-u.$ Write \[\phi_{t}(X)=tX+a_{1}X^{q}+a_{2}X^{q^{2}},\,\,\phi_{t-u}(X)=(t-u)X+a_{1}X^{q}+a_{2}X^{q^{2}}.\] 
Let $v$ be a finite prime.  
For the results in Sections~\ref{s212} and \ref{s222}, if we replace $t$ by $t-u$ and $\phi[t^{n}]$ by $\phi[(t-u)^{n}]$ for any $u\in \Bbb{F}_{q},$ all results establish since all claims rely only on the valuations of coefficients of $\phi_{t}(X).$ 
For $v$ being an infinite prime, if we replace $t$ by $t-u$ and $\phi[t^{n}]$ by $\phi[(t-u)^{n}]$ for any $u\in \Bbb{F}_{q},$ then the results in Sections~\ref{s211} and \ref{s221} establish.
Consequently, the proof of (1) below works for any $\pi$ and can straightforwardly imply (2). 

Next, we show (1) for the case where $v$ is infinite and $v(\bm{j})<v_{0}q.$ 
Let $m$ be the integer such that $v(\bm{j})\in (v_{0}q^{m+1},v_{0}q^{m}).$ 
By Theorem~\ref{p2211}, we know that any $\sigma\in G(K_{\infty}/K)^{y}=G(K_{m}/K)^{y}$ for $y>0$ fixes all $\xi_{-i}$ for $i\geq 1.$ 
Thus any nontrivial element $\sigma\in G(K_{m}/K)^{y}$ nontrivially acts on $(\xi_{n})_{n\geq 1}\in T_{\pi}$ and fixes $(\xi_{-n})_{n\geq 1}\in T_{\pi}.$ 
This implies $\text{\rm{rank}}_{\,\Bbb{F}_{q}[t]_{\pi}} \,T_{\pi}^{G_{v}^{y}}=1$ if $0<y\leq \frac{r_{1}}{E}$ and $=2$ if $\frac{r_{1}}{E}<y.$ 
Hence 
\[\mathfrak{f}_{v}(\phi)=\int_{0}^{\frac{r_{1}}{E}}1dy=\frac{-v(\bm{j})+v_{0}q}{q-1}.\]

We show (1) for the case where $v$ is finite and $v(\bm{j})<0.$ 
The absolute ramification subgroups $G_{v}^{y}$ for $y>0$ acts on $T_{\pi}$ via  $G(K_{\infty}/K)^{y}=\varprojlim_{n} G(K_{n}/K)^{y}.$ 
For any $n\geq 1,$ applying Theorem~\ref{p2221}, we know that any $\sigma\in G(K_{n}/K)^{y}$ fixes all $\xi_{-i}$ for $i=1, \ldots, n.$ 
So each nontrivial element of $G(K_{\infty}/K)^{\frac{r}{E}}$ nontrivially acts on $(\xi_{n})_{n\geq 1}\in T_{\pi}$ and fixes $(\xi_{-n})_{n\geq 1}\in T_{\pi}.$ 
Thus $\text{\rm{rank}}_{\,\Bbb{F}_{q}[t]_{\pi}} \,T_{\pi}^{G_{v}^{y}}=1$ if $0<y\leq \frac{r}{E}$ and $=2$ if $\frac{r}{E}<y.$ 
We have
\[\mathfrak{f}_{v}(\phi)=\int_{0}^{\frac{r}{E}}1dy=\frac{-v(\bm{j})}{q-1}.\]
Thus $\mathfrak{f}_{v}(\phi)$ is convergent. 

For the case where $v$ is infinite and $v(\bm{j})\geq v_{0}q,$ and the case where $v$ is finite and $v(\bm{j})\geq 0,$ Lemmas~\ref{l231} and \ref{l232} show that the extensions $K(\phi[\pi^{n}])/K$ are at worst tamely ramified for all $n\geq 1.$
Thus the $1$-conductors at these primes vanish.
Therefore the definition of $\mathfrak{f}_{v}(\phi)$ is valid. 
\end{proof}

We have also shown the following corollary.
\begin{corollary}\label{c311}
With the assumptions in Lemma-Definition~\text{\rm{\ref{d311}}}, we have
\begin{align}\mathfrak{f}_{v}(\phi)=\begin{cases}
\begin{cases}
0 & v(\bm{j})\in [v_{0}q,+\infty);\\
\frac{-v(\bm{j})+v_{0}q}{q-1}& v(\bm{j})\in(-\infty,v_{0}q),\,\,p\nmid v(\bm{j}),
\end{cases}&v\text{ is infinite};\\
\begin{cases}
0&v(\bm{j})\in [0,+\infty);\\
\frac{-v(\bm{j})}{q-1}&v(\bm{j})\in (-\infty,0),\,\,v(\bm{j})\nmid p,\text{ and }q \neq 2,
\end{cases}&v\text{ is finite}.
\end{cases}\nonumber
\end{align}
\end{corollary}

\subsection{A relation between \text{\it{J}}-heights and 1-conductors}\label{s312}
Fix a rank $2$ Drinfeld $\Bbb{F}_{q}[t]$-module $\phi$ defined over $F.$
In this subsection, we would like to consider a relation between a global invariant defined by $1$-conductors at all primes with some height of $\phi.$ 
By Corollary~\ref{c311}, the values of the $1$-conductors involve the valuations of the $j$-invariant and the $\pi.$
Thus the $J$-height defined in \cite{BPR} seems to be a good candidate for such a height. Let us recall its definition.

Let $M_{F}^{f}$ (resp.\ $M_{F}^{\infty}$) denote the set of all finite primes (resp.\ all infinite primes) of $F.$
Put $M_{F}\coloneqq M_{F}^{f}\cup M_{F}^{\infty}.$
For a prime $v$ of $F$ and the completion $F_{v},$ let $\text{deg}(v)$ denote the degree of the residue field of $F_{v}$ over $\Bbb{F}_{q}.$
Let $\bm{j}$ denote the $j$-invariant of $\phi.$ 
We adopt the definition of $J$-height of $\phi$ in our notation:
\begin{align}h_{J}(\phi)\coloneqq \frac{1}{[F:\Bbb{F}_{q}(t)]}\sum_{v\in M_{F}}\text{\rm{deg}}(v)\cdot\text{\rm{max}}\{-v(\bm{j}),\,0\}.\label{f412}
\end{align}
This invariant is independent of the choice of $F$ and Drinfeld modules in the isomorphism class. 

We follow the notation in Section~\ref{s311}.
Define the \text{\it{global}} $1$\text{\it{-conductor}} of the Drinfeld module $\phi$ by
\begin{align}\mathfrak{f}(\phi)\coloneqq \sum_{v\in M_{F}}\text{deg}(v)\cdot  \mathfrak{f}_{v}(\phi).\nonumber
\end{align}

The following claim is a consequence of Corollary~\ref{c311} and the definition of $J$-height.
\begin{theorem}[Functional $1$-Szpiro theorem]\label{p311}
Let $\phi$ be a rank $2$ Drinfeld $\Bbb{F}_{q}[t]$-module defined over $F$ such that its $j$-invariant $\bm{j}$ satisfies the condition \text{\rm{(\ref{f411})}} for any prime $v$ of $F.$ 
Then 
\begin{align}h_{J}(\phi)\leq \mathfrak{f}(\phi)\cdot\frac{q-1}{[F:\Bbb{F}_{q}(t)]}+q.\label{f336}\end{align}
\end{theorem} 
Under the conditions in the theorem, the equality in (\ref{f336}) holds if and only if there is no infinite prime $v$ of $F$ such that $v(\bm{j})\geq v_{0}q.$ 
\begin{proof}

We know from Corollary~\ref{c311} that 
\begin{align}\mathfrak{f}(\phi)=&\left(\sum_{v\in M_{F}^{f}}\text{deg}(v)\cdot \text{max}\left\{\frac{-v(\bm{j})}{q-1},\,0\right\}\right)\nonumber\\
&+\left(\sum_{v\in M_{F}^{\infty}}\text{deg}(v)\cdot \text{max}\left\{\frac{-v(\bm{j})+v_{0}q}{q-1},\,0\right\}\right).\nonumber\end{align}
We have
\begin{align}&\frac{q-1}{[F:\Bbb{F}_{q}(t)]}\mathfrak{f}(\phi)-h_{J}(\phi)\nonumber\\
=&\frac{1}{[F:\Bbb{F}_{q}(t)]}\sum_{v\in M_{F}^{\infty}}\text{deg}(v)\left(\text{max}\left\{-v(\bm{j})+v_{0}q,0\right\}-\text{max}\left\{-v(\bm{j}),0\right\}\right)\nonumber\\
\geq & \frac{1}{[F:\Bbb{F}_{q}(t)]}\sum_{v\in M_{F}^{\infty}}\text{\rm{deg}}(v)\cdot v_{0}q=\frac{1}{[F:\Bbb{F}_{q}(t)]}\sum_{v\in M_{F}^{\infty}}f_{v}\cdot e_{v}\cdot  (-\text{deg}(\pi)q)=-q,\nonumber
\end{align}
where we use the extension formula in the last equality. This shows the theorem.
\end{proof}

\begin{remark}\label{r312}
Although the conditions in the theorem seem strict, it is still not hard to find infinitely many rank $2$ Drinfeld modules fulfilling these conditions. 
Consider for each $i\geq 2$ with $p\nmid i$ the Drinfeld $\Bbb{F}_{q}[t]$-module $\phi^{(i)}$ over $\Bbb{F}_{q}(t)$ defined by $\phi_{t}^{(i)}(X)\coloneqq tX+t^{i}X^{q}+X^{q^{2}}.$ It has good reduction at all finite primes. The $j$-invariant of $\phi^{(i)}$ is $\bm{j}(\phi^{(i)})=t^{i(q+1)}.$ 
Thus at the only infinite prime $v,$ the $1$-conductor $\mathfrak{f}_{v}(\phi^{(i)})=\frac{i(q+1)-q}{q-1}$ of $\phi^{(i)}$ can be obtained.
This family is an example that the conductor can be arbitrarily large.
Note that the equality in (\ref{f336}) in the theorem holds.
\end{remark}

\appendix
\section{At primes lying above $t$}
\label{s213}
We are to calculate the valuations of the $t^{n}$-division points of a rank $2$ Drinfeld $\Bbb{F}_{q}[t]$-module $\phi$ over a local field corresponding to a prime of $F$ lying above $t.$ 
The method is similar to that in Section~\ref{s21}. 
There is no application of this appendix in this paper.

Let $v_{t}$ denote a prime of $F$ dividing $t$ and let it also denote the normalized valuation corresponding to $v_{t}.$
We have $v_{0}\coloneqq v_{t}(t)>0.$ Put $v_{1}\coloneqq v_{t}(a_{1})$ and $v_{2}\coloneqq v_{t}(a_{2})$ with $a_{1},$ $a_{2}$ in $\phi_{t}(X)=tX+a_{1}X^{q}+a_{2}X^{q^{2}}.$ 
Put $P_{0}=(1,v_{0}),$ $P_{1}=(q,v_{1}),$ and $P_{2}=(q^{2},v_{2}).$
We calculate 
\begin{align}\mu(P_{0},P_{1})-\mu(P_{0},P_{2})=\frac{v_{t}(\bm{j})-v_{0}q}{q^{2}-1}.
\nonumber
\end{align}
Assume $v_{t}(\bm{j})<v_{0}q.$ 
The Newton polygon of $\phi_{t}(X)$ is $P_{0}P_{1}P_{2}$ having exactly two segments. 
Let $\xi_{-1}$ and $\xi_{1}$ be two roots of $\phi_{t}(X)$ with valuations 
\[v_{t}(\xi_{-1})=-\mu(P_{0},P_{1})=-\frac{v_{1}-v_{0}}{q-1},\,\,v_{t}(\xi_{1})=-\mu(P_{1},P_{2})=-\frac{v_{2}-v_{1}}{(q-1)q}.\]
In general, for $n\geq 1,$ denote by $\xi_{-(n+1)}$ (resp.\ $\xi_{n+1}$) a root of $\phi_{t}(X)=\xi_{-n}$ (resp.\ $\phi_{t}(X)=\xi_{n}$) such that $v_{t}(\xi_{-(n+1)})$ (resp.\ $v_{t}(\xi_{n+1})$) is the largest among the valuations of all roots of $\phi_{t}(X)=\xi_{-n}$ (resp.\ $\phi_{t}(X)=\xi_{n}$).
The next lemma concerns the valuations of $\xi_{-n}$ and $\xi_{n}.$

\begin{lemma}
\begin{enumerate}[\rm{(}1)]
\item
Assume $0<v_{t}(\bm{j})<v_{0}q.$ 
Let $m$ be the integer such that $v_{t}(\bm{j})\in [v_{0}/q^{m-1},v_{0}/q^{m-2}).$
Put $Q_{-n}=(0,v_{t}(\xi_{-n}))$ and $Q_{n}=(0,v_{t}(\xi_{n})).$
\begin{enumerate}[\rm{(}i)]
\item We have
\[v_{t}(\xi_{-n})=\begin{cases}
-\dfrac{-v_{0}+v_{1}q^{n-1}}{(q-1)q^{n-1}} & 1\leq n\leq m;\\
-\dfrac{-v_{0}+v_{1}q^{m-1}+v_{2}q^{m-1}(q^{2n-2m}-1)/(q+1)}{(q-1)q^{2n-m-1}} & n\geq m+1.
\end{cases}\]
If $n\leq m-1,$ then the Newton polygon of  $\phi_{t}(X)-\xi_{-n}$ is $Q_{-n}P_{1}P_{2}$ having exactly two segments. 
If $n\geq m,$ then the Newton polygon of $\phi_{t}(X)-\xi_{-n}$ is $Q_{-n}P_{2}$ having exactly one segment. 

\item For $n\geq 1,$ we have 
\[v_{t}(\xi_{n})=-\frac{-v_{1}+
v_{2}(q^{2n-1}+1)/(q+1)}{(q-1)q^{2n-1}}.\]
The Newton polygon of 
$\phi_{t}(X)-\xi_{n}$ is $Q_{n}P_{2}$ having exactly one segment. 
\end{enumerate}

\item Assume $v_{t}(\bm{j})\leq 0.$
We have
\begin{align}v_{t}(\xi_{-n})&=-\frac{-v_{0}+v_{1}q^{n-1}}{(q-1)q^{n-1}},\nonumber\\v_{t}(\xi_{n})&=-\frac{v_{2}+v_{1}(q^{n}-q-1)}{(q-1)q^{n}}.\nonumber\end{align}
If $v_{t}(\bm{j})<0,$ then the Newton polygon of $\phi_{t}(X)-\xi_{-n}$ \text{\rm{(}}resp.\ $\phi_{t}(X)=\xi_{n}$\text{\rm{)}} is $Q_{-n}P_{1}P_{2}$ \text{\rm{(}}resp.\ $Q_{n}P_{1}P_{2}$\text{\rm{)}} having exactly two segments. 
If $v_{t}(\bm{j})=0,$ then the Newton polygon of $\phi_{t}(X)-\xi_{-n}$ \text{\rm{(}}resp.\ $\phi_{t}(X)-\xi_{n}$\text{\rm{)}} is $Q_{-n}P_{1}P_{2}$ \text{\rm{(}}resp.\ $Q_{n}P_{2}$\text{\rm{)}} having exactly two segments \text{\rm{(}}resp.\ one segment\text{\rm{)}}.
\end{enumerate}
\end{lemma}

\begin{proof}
We prove the lemma by induction on $n.$ 
We determine the Newton polygon of $\phi_{t}(X)-\xi_{1}$ first.
We have 
\begin{align}\mu(Q_{-1},P_{0})=v_{0}+\frac{-v_{0}+v_{1}}{q-1},\,\,\mu(Q_{-1},P_{1})=\frac{v_{1}+\frac{-v_{0}+v_{1}}{q-1}}{q-0},\,\,\mu(Q_{-1},P_{2})=\frac{v_{2}+\frac{-v_{0}+v_{1}}{q-1}}{q^{2}-0}.\label{fa5}\end{align}
We calculate 
\begin{align}
\mu(Q_{-1},P_{0})-\mu(Q_{-1},P_{1})&=\frac{v_{0}(q-1)}{q}>0,\label{fa6}\\
\mu(Q_{-1},P_{1})-\mu(Q_{-1},P_{2})&=\frac{v_{t}(\bm{j})-v_{0}}{q^{2}}.\label{fa7}
\end{align}
If $v_{t}(\bm{j})\in [v_{0},v_{0}q),$ then we have $\mu(Q_{-1},P_{1})\geq \mu(Q_{-1},P_{2})$ and the Newton polygon of $\phi_{t}(X)-\xi_{-1}$ is $Q_{-n}P_{2}$ having exactly one segment.
If $v_{t}(\bm{j})< v_{0},$ then we have $\mu(Q_{-1},P_{1})<\mu(Q_{-1},P_{2})$ and the Newton polygon is $Q_{-1}P_{1}P_{2}$ having exactly two segments. 

Assume that (i) for $n-1$ is valid.
If $n\leq m-1,$ the valuation of $\xi_{-n}$ is
\[-\mu(Q_{-(n-1)},P_{1})=-\frac{v_{1}-v_{t}(\xi_{-(n-1)})}{q-0}=-\frac{-v_{0}+v_{1}q^{n-1}}{(q-1)q^{n-1}}.\]
Next, we determine the Newton polygon of $\phi_{t}(X)-\xi_{-n}.$
We have 
\begin{align}\mu(Q_{-n},P_{0})=v_{0}-v_{t}(\xi_{-n}),\,\,\mu(Q_{-n},P_{1})=\frac{v_{1}-v_{t}(\xi_{-n})}{q-0},\,\,\mu(Q_{-n},P_{2})=\frac{v_{2}-v_{t}(\xi_{-n})}{q^{2}-0}.\label{fa4}\end{align}
We calculate 
\begin{align}
\mu(Q_{-n},P_{0})-\mu(Q_{-n},P_{1})&=\frac{v_{0}(q^{n}-1)}{q^{n}}>0,\label{fa2}\\
\mu(Q_{-n},P_{1})-\mu(Q_{-n},P_{2})&=\frac{v_{t}(\bm{j})q^{n-1}-v_{0}}{q^{n+1}}.\label{fa3}
\end{align}
Since $n\leq m-1,$ we have $\mu(Q_{-n},P_{1})<\mu(Q_{-n},P_{2}).$
This implies that $Q_{-n}P_{1}$ is the first segment of the Newton polygon and the Newton polygon is $Q_{-n}P_{1}P_{2}$ having exactly two segments.

When $n=m,$ we have the same inductive hypothesis as above and $v_{\infty}(\xi_{m})=-\mu(Q_{-(m-1)},P_{1}).$
However, we have $\mu(Q_{-m},P_{1})\geq \mu(Q_{-m},P_{2})$ by (\ref{fa3}).
By (\ref{fa2}), the Newton polygon of $\phi_{t}(X)-\xi_{-m}$ is $Q_{-m}P_{2}$ having exactly one segment.

If $n\geq m+1,$ then the valuation of $\xi_{-n}$ is calculated by $-\mu(Q_{-n},P_{2})=-(v_{2}-v_{t}(\xi_{-n}))/q^{2}.$
We show that the Newton polygon of $\phi_{t}(X)-\xi_{-n}$ is $Q_{-n}P_{2}$ having exactly one segment.
We have $\mu(Q_{-n},P_{0}),$ $\mu(Q_{-n},P_{1}),$ and $\mu(Q_{-n},P_{2})$ as in (\ref{fa4}).
The Newton polygon is determined by the inequalities
\begin{align}
\mu(Q_{-n},P_{0})-\mu(Q_{-n},P_{2})&=\frac{v_{t}(\bm{j})q^{m-1}+v_{0}(q^{2n-m+1}-q-1)}{q^{2n-m+1}}>0,\nonumber\\
\mu(Q_{-n},P_{1})-\mu(Q_{-n},P_{2})&=\frac{v_{t}(\bm{j})(q^{2n-m}+q^{m-1})/(q+1)-v_{0}}{q^{2n-m+1}}> \frac{v_{t}(\bm{j})q^{m-1}-v_{0}}{q^{2n-m+1}}\geq 0.\nonumber
\end{align}

For (ii) of (1), we first show that the Newton polygon of $\phi_{t}(X)-\xi_{1}$ is $Q_{1}P_{2}$ having exactly one segment.
We have 
\begin{align}\mu(Q_{1},P_{0})=v_{0}+\frac{-v_{1}+v_{2}}{(q-1)q},\,\,\mu(Q_{1},P_{1})=\frac{v_{1}+\frac{-v_{1}+v_{2}}{(q-1)q}}{q-0},\,\,\mu(Q_{1},P_{2})=\frac{v_{2}+\frac{-v_{1}+v_{2}}{(q-1)q}}{q^{2}-0}.\label{fa8}\end{align}
The Newton polygon is determined by the inequalities
\begin{align}
\mu(Q_{1},P_{0})-\mu(Q_{1},P_{2})&=\frac{-v_{t}(\bm{j})+v_{0}q^{3}}{q^{3}}>0,\label{fa9}\\
\mu(Q_{1},P_{1})-\mu(Q_{1},P_{2})&=\frac{v_{t}(\bm{j})(q-1)}{q^{3}}>0.\label{fa10}
\end{align}
Assume that (ii) for $n-1$ is valid.
The valuation of $\xi_{n}$ is calculated by $-\mu(Q_{n-1},P_{2})=-(v_{2}-v_{t}(\xi_{n-1}))/q^{2}.$
We have 
\begin{align}\mu(Q_{n},P_{0})=v_{0}-v_{t}(\xi_{n}),\,\,\mu(Q_{n},P_{1})=\frac{v_{1}-v_{t}(\xi_{n})}{q-0},\,\,\mu(Q_{n},P_{2})=\frac{v_{2}-v_{t}(\xi_{n})}{q^{2}-0}.\label{fa11}\end{align}
We calculate 
\begin{align}
\mu(Q_{n},P_{0})-\mu(Q_{n},P_{2})&=\frac{-v_{t}(\bm{j})+v_{0}q^{2n+1}}{q^{2n+1}}>0,\nonumber\\
\mu(Q_{n},P_{1})-\mu(Q_{n},P_{2})&=\frac{v_{t}(\bm{j})(q^{2n}-1)/(q+1)}{q^{2n+1}}>0.\nonumber
\end{align}
Thus the Newton polygon of $\phi_{t}(X)-\xi_{n}$ is $Q_{n}P_{2}$ having exactly one segment.

Next, we show the result in (2) for $\xi_{-n}.$
Assume $v_{t}(\bm{j})\leq 0.$
We determine the Newton polygon of $\phi_{t}(X)-\xi_{-1}.$
We have $\mu(Q_{-1},P_{0}),$ $\mu(Q_{-1},P_{1}),$ and $\mu(Q_{-1},P_{2})$ as in (\ref{fa5}).
Since $v_{t}(\bm{j})\leq 0,$ we have  $\mu(Q_{-1},P_{1})<\mu(Q_{-1},P_{2})$ by (\ref{fa7}).
Then (\ref{fa6}) implies that the Newton polygon of $\phi_{t}(X)-\xi_{-n}$ is $Q_{-1}P_{1}P_{2}$ having exactly two segments.
Assume that (2) for $n-1$ is valid. 
Then the valuation of $\xi_{-n}$ is calculated by $-\mu(Q_{-(n-1)},P_{1})=-(v_{1}-v_{t}(\xi_{-(n-1)}))/q.$
To determine the Newton polygon of $\phi_{t}(X)-\xi_{-n},$ we compare the slopes given in (\ref{fa4}).
Since $v_{t}(\bm{j})\leq 0,$ we have $\mu(Q_{-n},P_{1})<\mu(Q_{-n},P_{2})$ by (\ref{fa3}). 
Then (\ref{fa2}) implies that the Newton polygon is $Q_{-n}P_{1}P_{2}$ having exactly two segments.

Finally, we show the result in (2) for $\xi_{n}.$
To determine the Newton polygon of $\phi_{t}(X)-\xi_{1},$ we compare the slopes given in (\ref{fa8}).
If $v_{t}(\bm{j})<0,$ we have the inverse of the inequality (\ref{fa10}), i.e., $\mu(Q_{1},P_{1})<\mu(Q_{1},P_{2}).$ 
Then (\ref{fa9}) implies that the Newton polygon is $Q_{1}P_{1}P_{2}$ having exactly two segments.
If $v_{t}(\bm{j})=0,$ we have
$\mu(Q_{1},P_{1})=\mu(Q_{1},P_{2}).$
Then (\ref{fa9}) implies that the Newton polygon is $Q_{1}P_{2}$ having exactly one segment.
Assume that (2) for $n-1$ is valid.
Assume $v_{t}(\bm{j})<0.$ 
The valuation of $\xi_{n}$ is calculated by $-\mu(Q_{n-1},P_{1})=-(v_{1}-v_{t}(\xi_{n-1}))/q.$
We show that the Newton polygon of $\phi_{t}(X)-\xi_{n}$ is $Q_{n}P_{1}P_{2}$ having exactly two segments.
We have $\mu(Q_{n},P_{0}),$ $\mu(Q_{n},P_{1}),$ and $\mu(Q_{n},P_{2})$ as in (\ref{fa11}).
The Newton polygon is determined by the inequalities
\begin{align}
\mu(Q_{n},P_{0})-\mu(Q_{n},P_{1})&=\frac{-v_{t}(\bm{j})+v_{0}q^{n+1}}{q^{n+1}}>0,\nonumber\\
\mu(Q_{n},P_{1})-\mu(Q_{n},P_{2})&=\frac{v_{t}(\bm{j})(q^{n}-1)}{q^{n+2}}<0.\nonumber
\end{align}
If $v_{t}(\bm{j})=0,$ we have $v_{1}(q+1)=v_{2}$ and $v_{t}(\xi_{n-1})=-v_{1}/(q-1)$ for any $n.$ 
This implies that the valuation of $\xi_{n}$ is $-\mu(Q_{n-1},P_{1})=-\mu(Q_{n-1},P_{2})=-v_{1}/(q-1).$
As for the Newton polygon of $\phi_{t}(X)-\xi_{n}$, we have $\mu(Q_{n},P_{0}),$ $\mu(Q_{n},P_{1}),$ and $\mu(Q_{n},P_{2})$ as in (\ref{fa11}).
We calculate $\mu(Q_{n},P_{0})-\mu(Q_{n},P_{2})=v_{0}>0$ and $\mu(Q_{n},P_{1})=\mu(Q_{n},P_{2}).$
Therefore the Newton polygon is $Q_{n}P_{2}$ having exactly one segment.
\end{proof} 

The next result concerns the case $v_{t}(\bm{j})\geq v_{0}q.$ 
Now all nonzero roots of $\phi_{t}(X)$ have valuation
$-\frac{v_{2}-v_{0}}{q^{2}-1}.$
Let $\xi_{-1}$ and $\xi_{1}$ be elements of $\phi[t]$ such that they form a basis. 
Define $\xi_{-n}$ and $\xi_{n}$ inductively as in the paragraph before Proposition~\ref{p2112}.

\begin{lemma}
Assume $v_{t}(\bm{j})\in [v_{0}q,+\infty).$
We have
\[v_{t}(\xi_{-n})=v_{t}(\xi_{n})=-\frac{-v_{0}+v_{2}q^{2n-2}}{(q^{2}-1)q^{2n-2}}.\]
Put $Q_{n}=(0,v_{t}(\xi_{-n}))=(0,v_{t}(\xi_{n})).$
Then the Newton polygons of 
$\phi_{t}(X)-\xi_{-n}$ and $\phi_{t}(X)-\xi_{n}$ are $Q_{n}P_{2}$ having exactly one segment.
\end{lemma}

\begin{proof}
We prove the lemma by induction on $n.$ 
We first check the Newton polygons of $\phi_{t}(X)-\xi_{\pm 1}$ (i.e., the polynomials $\phi_{t}(X)-\xi_{-1}$ and $\phi_{t}(X)-\xi_{1}$) are $Q_{1}P_{2}$ having exactly one segment.
We have
\[\mu(Q_{1},P_{0})=v_{0}+\frac{-v_{0}+v_{2}}{q^{2}-1},\,\,\mu(Q_{1},P_{1})=\frac{v_{1}+\frac{-v_{0}+v_{2}}{q^{2}-1}}{q-0},\,\,\mu(Q_{1},P_{2})=\frac{v_{2}+\frac{-v_{0}+v_{2}}{q^{2}-1}}{q^{2}-0}.\]
The Newton polygons are determined by the inequalities 
\begin{align}
\mu(Q_{1},P_{0})-\mu(Q_{1},P_{2})&=\frac{v_{0}(q^{2}-1)}{q^{2}}>0,\nonumber\\
\mu(Q_{1},P_{1})-\mu(Q_{1},P_{2})&=\frac{v_{t}(\bm{j})q-v_{0}}{(q+1)q^{2}}>0.\nonumber
\end{align}
Assume (1) for $n-1.$
Then the valuations of $\xi_{\pm n}$ are calculated by $-\mu(Q_{n},P_{2})=-(v_{2}-v_{t}(\xi_{n}))/q^{2}.$
We show that the Newton polygons of $\phi_{t}(X)-\xi_{\pm n}$ are $Q_{n}P_{2}$ having exactly one segment.
We have 
\[\mu(Q_{n},P_{0})=v_{0}-v_{t}(\xi_{\pm n}), \mu(Q_{n},P_{1})=\frac{v_{1}-v_{t}(\xi_{\pm n})}{q-0}, \mu(Q_{n},P_{2})=\frac{v_{2}-v_{t}(\xi_{\pm n})}{q^{2}-0}.\]
The Newton polygons are determined by the inequalities
\begin{align}
\mu(Q_{n},P_{0})-\mu(Q_{n},P_{2})&=\frac{v_{0}(q^{2n}-1)}{q^{2n}}>0,\nonumber\\
\mu(Q_{n},P_{1})-\mu(Q_{n},P_{2})&=\frac{v_{t}(\bm{j})q^{2n-1}-v_{0}}{(q+1)q^{2n}}>0.\nonumber
\end{align}
\end{proof}

Similar to Proposition~\ref{p2111}, the next proposition claims that $\xi_{i}$ for $i\in \mathcal{I}_{n}$ form a basis of $\phi[t^{n}]$ and can be arranged with respect to their valuations.
Its proof is very similar to that of Proposition~\ref{p2111}.
For instance, if $v_{t}(\bm{j})\in [v_{0}/q^{m-1},v_{0}/q^{m-2})$ for a positive integer $m,$ we have for $i\geq 1$ that
\begin{itemize}
\item
$v_{t}(\xi_{-m-(i-1)})>v_{t}(\xi_{i})$ if and only if $v_{t}(\bm{j})<v_{0}/q^{m-2}$;
\item
$v_{t}(\xi_{i})\geq v_{t}(\xi_{-m-i})$ if and only if $v_{t}(\bm{j})\geq v_{0}/q^{m-1}.$
\end{itemize}
We omit the proof.

\begin{proposition}\label{pa2}
\begin{enumerate}[\rm{(}1)]
\item
Assume $0<v_{t}(\bm{j})<v_{0}q.$ 
Let $m$ be the integer such that $v_{t}(\bm{j})\in[v_{0}/q^{m-1},v_{0}/q^{m-2}).$
The roots $\xi_{j}$ for $j\in \mathcal{I}_{n}$ form a basis of $\phi[t^{n}].$ 
If $n\leq m,$ we have
\[v_{t}(\xi_{-1})>v_{t}(\xi_{-2})>\cdots>v_{t}(\xi_{-n})>v_{t}(\xi_{1})>v_{t}(\xi_{2})>\cdots>v_{t}(\xi_{n}).\]
If $n\geq m+1,$ we have 
\begin{align}&v_{t}(\xi_{-1})>v_{t}(\xi_{-2})>\cdots>v_{t}(\xi_{-m})\nonumber\\
>&v_{t}(\xi_{1})\geq v_{t}(\xi_{-m-1})\nonumber\\
>&v_{t}(\xi_{2})\geq v_{t}(\xi_{-m-2})\nonumber\\
>&\cdots\nonumber\\
>&v_{t}(\xi_{n-m})\geq v_{t}(\xi_{-n})\nonumber\\
>&v_{t}(\xi_{n-m+1})>v_{t}(\xi_{n-m+2})>\cdots>v_{t}(\xi_{n}),\nonumber\end{align}
where each equality holds if and only if $v_{t}(\bm{j})=v_{0}/q^{m-1}.$ 

\item
Assume $v_{t}(\bm{j})\leq 0.$ 
The roots $\xi_{j}$ for $j\in \mathcal{I}_{n}$ form a basis of $\phi[t^{n}].$
We have
\begin{align}v_{t}(\xi_{-1})> v_{t}(\xi_{-2})>\cdots>v_{t}(\xi_{-n})>v_{t}(\xi_{n})\geq v_{t}(\xi_{n-1})\geq \cdots\geq v_{t}(\xi_{1})\nonumber,\end{align}
where each equality holds if and only if $v_{t}(\bm{j})=0.$ 

\item 
Assume $v_{t}(\bm{j})\geq v_{0}q.$
The roots $\xi_{j}$ for $j\in \mathcal{I}_{n}$ form a basis of $\phi[t^{n}].$
We have
\[v_{t}(\xi_{-1})=v_{t}(\xi_{1})>v_{t}(\xi_{-2})=v_{t}(\xi_{2})>\cdots>v_{t}(\xi_{-n})=v_{t}(\xi_{n}).\]

\item
A root of $\phi_{t^{n}}(X)$ having valuation $v_{t}(\xi_{i})$ for some $i\in \mathcal{I}_{n}$ is an $\Bbb{F}_{q}$-linear combination of $\xi_{i'}$ such that 
\begin{itemize}
    \item each $\xi_{i'}$ has valuation equal to or larger than $v_{t}(\xi_{i});$
    \item at least one of the coefficients of $\xi_{i'}$ with $v_{t}(\xi_{i'})=v_{t}(\xi_{i})$ is nonzero.
\end{itemize}
\end{enumerate}
\end{proposition}

\begin{remark}
Let $\pi$ be a prime of $\Bbb{F}_{q}(t)$ and $K$ be the completion of $\Bbb{F}_{q}(t)$ at $\pi.$
For a Drinfeld module $\phi$ over $\Bbb{F}_{q}(t),$ if $\phi$ has good and supersingular reduction at $\pi$ 
\footnote{$\phi$ is said to have good and supersingular reduction at $\pi$ if the reduction of $\phi$ at $\pi$ is identified with a power of the Frobenius. See more information in \cite[3.1]{GP}.}, 
the ramification subgroups and the ramification breaks of the extension $K(\phi[\pi^{n}])/K$ (up to a compositum with an unramified extension of $K$) are worked out by Galateau and Pacheco~\cite[3.5]{GP}.

It would be interesting to consider whether there are similar results for a Drinfeld module over $\Bbb{F}_{p}(t)$ having \text{\it{potentially}} good and supersingular reduction at the prime $\pi.$ 
So can the above $K$ be replaced by its tamely ramified extensions?
This can be a nontrivial problem since Neukirch's generalization of the cyclotomic theory (in the setting of the Lubin-Tate theory) seems to only allow $K$ replaced by its unramified extensions.
\end{remark}

\bibliographystyle{amsalpha}
\bibliography{literature/library}

\providecommand{\bysame}{\leavevmode\hbox to3em{\hrulefill}\thinspace}
\providecommand{\MR}{\relax\ifhmode\unskip\space\fi MR }
\providecommand{\MRhref}[2]{%
  \href{http://www.ams.org/mathscinet-getitem?mr=#1}{#2}
}
\providecommand{\href}[2]{#2}
\begin{thebibliography}{BPR21}

\bibitem[Bir67]{Bir}
Bryan~J. Birch, \emph{Cyclotomic fields and {K}ummer extensions}, Algebraic
  {N}umber {T}heory ({P}roc. {I}nstructional {C}onf., {B}righton, 1965),
  Academic Press, London, 1967, pp.~85--93. \MR{219507}

\bibitem[BK94]{BK}
Armand Brumer and Kenneth Kramer, \emph{The conductor of an abelian variety},
  Compositio Math. \textbf{92} (1994), no.~2, 227--248. \MR{1283229}

\bibitem[BPR21]{BPR}
Florian Breuer, Fabien Pazuki, and Mahefason~H. Razafinjatovo, \emph{Heights
  and isogenies of {D}rinfeld modules}, Acta Arith. \textbf{197} (2021), no.~2,
  111--128. \MR{4189716}

\bibitem[FV02]{FV}
Ivan~B. Fesenko and Serge\u{\i}~V. Vostokov, \emph{Local fields and their
  extensions}, second ed., Translations of Mathematical Monographs, vol. 121,
  American Mathematical Society, Providence, RI, 2002, With a foreword by I. R.
  Shafarevich. \MR{1915966}

\bibitem[GP17]{GP}
Aur\'{e}lien Galateau and Am\'{\i}lcar Pacheco, \emph{Hauteur et torsion des
  modules de {D}rinfeld de rang $2$}, J. Number Theory \textbf{179} (2017),
  185--219. \MR{3657163}

\bibitem[KS04]{KL}
Michael K\"{o}lle and Peter Schmid, \emph{Computing {G}alois groups by means of
  {N}ewton polygons}, Acta Arith. \textbf{115} (2004), no.~1, 71--84.
  \MR{2102807}

\bibitem[LRS93]{LRS}
Paul Lockhart, Michael Rosen, and Joseph~H. Silverman, \emph{An upper bound for
  the conductor of an abelian variety}, J. Algebraic Geom. \textbf{2} (1993),
  no.~4, 569--601. \MR{1227469}

\bibitem[Mau19]{Mau}
Andreas Maurischat, \emph{On field extensions given by periods of {D}rinfeld
  modules}, Arch. Math. (Basel) \textbf{113} (2019), no.~3, 247--254.
  \MR{3988819}

\bibitem[Moc21]{Mo}
Shinichi Mochizuki, \emph{Inter-universal {T}eichm\"{u}ller theory {IV}:
  {L}og-volume computations and set-theoretic foundations}, Publ. Res. Inst.
  Math. Sci. \textbf{57} (2021), no.~1-2, 627--723. \MR{4225476}

\bibitem[Neu99]{Neu}
J\"{u}rgen Neukirch, \emph{Algebraic number theory}, Grundlehren der
  ma\-the\-ma\-ti\-schen Wissenschaften [Fundamental Principles of Mathematical
  Sciences], vol. 322, Springer-Verlag, Berlin, 1999, Translated from the 1992
  German original and with a note by Norbert Schappacher, With a foreword by G.
  Harder. \MR{1697859}

\bibitem[Ser79]{Se}
Jean-Pierre Serre, \emph{Local fields}, Graduate Texts in Mathematics, vol.~67,
  Springer-Verlag, New York-Berlin, 1979, Translated from the French by Marvin
  Jay Greenberg. \MR{554237}

\bibitem[Sil94]{Sil}
Joseph~H. Silverman, \emph{Advanced topics in the arithmetic of elliptic
  curves}, Graduate Texts in Mathematics, vol. 151, Springer-Verlag, New York,
  1994. \MR{1312368}

\bibitem[Szp90]{Sz}
Lucien Szpiro, \emph{Discriminant et conducteur des courbes elliptiques}, no.
  183, 1990, S\'{e}minaire sur les Pinceaux de Courbes Elliptiques (Paris,
  1988), pp.~7--18. \MR{1065151}

\end{thebibliography}

$\textsc{Center for Innovative Teaching and Learning, }\\\textsc{\,\,\,\,\,\,\,\,Tokyo Institute of Technology, }\\\textsc{\,\,\,\,\,\,\,\,2-12-1, O-okayama, Meguro-ku, Tokyo 152-8550, Japan.}$

\,

$\textsc{Department of Mathematics, }\\\textsc{\,\,\,\,\,\,\,\,New Uzbekistan University,}\\\textsc{\,\,\,\,\,\,\,\,Movarounnahr street 1, Tashkent 100001, Uzbekistan.}$
\end{document}